\documentclass[letterpaper, 11pt]{amsart}

\usepackage{fullpage}
\usepackage{graphicx}
\usepackage{amsmath, amssymb, amsfonts, amsthm}
\usepackage{url}
\usepackage{amsaddr}

\usepackage[ruled,vlined]{algorithm2e}

\usepackage{color}
\usepackage{subfigure}
\definecolor{darkblue}{rgb}{0,0,0.4}
\usepackage[plainpages=false,colorlinks=true,citecolor=blue,filecolor=black,linkcolor=red,urlcolor=darkblue]{hyperref}

\newtheorem{thm}{Theorem}[section]
\newtheorem{lem}[thm]{Lemma}
\theoremstyle{remark}
\newtheorem{rem}[thm]{Remark}

\newcommand{\ie}{{\it i.e.}}
\newcommand{\eg}{{\it e.g.}}
\newcommand{\wto}{\rightharpoonup}

\usepackage[backend=bibtex,giveninits=true,style=alphabetic,maxbibnames=9,natbib=true,url=false,sorting=nyt,doi=true,backref=false]{biblatex}
\renewbibmacro{in:}{\ifentrytype{article}{}{\printtext{\bibstring{in}\intitlepunct}}}

\begin{document}
\title{Diffusion generated methods for \\ denoising target-valued images}

\author{Braxton Osting and Dong Wang}
\address{Department of Mathematics, University of Utah, Salt Lake City, UT}
\email{\{osting,dwang\}@math.utah.edu}
\thanks{B. Osting is partially supported by NSF DMS 16-19755 and DMS 17-52202. }




\date{\today}

\begin{abstract} 
We consider the inverse problem of denoising an image where each point (pixel) is an element of a target set, which we refer to as a target-valued image. The target sets considered are either (i) a closed convex set of Euclidean space or (ii) a closed subset of the sphere such that the closest point mapping is defined almost everywhere. The energy for the denoising problem consists of an $L^2$-fidelity term which is regularized by the Dirichlet energy. A relaxation of this energy, based on the heat kernel, is introduced and the associated minimization problem is proven to be well-posed. We introduce a diffusion generated method which can be used to efficiently find minimizers of this energy. We prove results for the stability and convergence of the method for both types of target sets.  The method is demonstrated on a variety of synthetic and test problems, with associated target sets given by the semi-positive definite matrices, the cube, spheres, the orthogonal matrices, and the real projective line. 
\end{abstract}

\maketitle

\section{Introduction}
Let $\Omega \subset \mathbb R^d$ be a Euclidean set with smooth boundary. We consider noisy \emph{target-valued data}, $f \colon \Omega \to \mathbb R^k$, that takes values in (or near) a certain \emph{target set}, $T \subset \mathbb R^k$. We assume that either 
\begin{enumerate}
\item[(i)] $T\subset \mathbb R^k$ is a closed convex set or 
\item[(ii)] $T\subset \mathbb R^k$ is a closed subset of the unit sphere, $\mathbb S^{k-1}$, such that the closest point mapping, $\Pi_T$, is defined almost everywhere; see Section \ref{s:TargetSetSphere}. 
\end{enumerate} 
Our goal will be to find a smooth map, $u \colon \Omega \to T$, that approximates the data, $f$. For $\alpha > 0$, we consider the general inverse problem, 
\begin{equation} \label{eq:original}
\inf_{u \colon \Omega \to T } \ E_\alpha(u), 
\qquad \qquad \textrm{where} \quad
E_\alpha(u)  = \frac{1}{2} \int_{\Omega} | \nabla u(x) |^2  \ dx  + \frac{1}{2 \alpha} \int_{\Omega} |u(x) - f(x)|^2 \ dx. 
\end{equation}
For a vector valued function $u$, the gradient in $E_\alpha$ should be interpreted component-wise. 
The parameter $\alpha$ controls the tradeoff between the `smoothness of $u$' and the `fidelity to the data'; in the limit  $\alpha \searrow 0$, we obtain $u \equiv f$ on $\Omega$. For some target sets, $T$, it is difficult, either analytically or computationally, to handle the constraint that $u$ take values in $T$. (This is why we didn't specify the class of functions $u \colon \Omega \to T$ to take the infimum in \eqref{eq:original}.)

\subsection*{One penalization approach} 
Suppose, for the target set $T$, there exists a smooth \emph{auxiliary function}, $L \colon \mathbb R^k \to \mathbb R_+$, such that $T = L^{-1}(0)$, \ie, $T$ is the zero-level set and set of global minimizers of the non-negative function $L$. In this case, since $T = \arg\min_x L(x) \subset \mathbb R^k$, we can use the function $L$ to \emph{penalize} when $u$ does not take values in $T$. For $\varepsilon >0$, one may consider the relaxation of \eqref{eq:original}, 
\begin{equation}  
\label{eq:relax}
\inf_{u \in H^1(\Omega;\mathbb R^k )}  \ F_{\alpha,\varepsilon}(u), 
\quad \textrm{where} \ 
F_{\alpha,\varepsilon}(u) =  \frac{1}{2} \int_\Omega  |\nabla u(x) |^2 \ dx+ \frac{1}{2 \alpha} \int_{\Omega} |u(x) - f(x)|^2 \ dx + \ \frac{1}{\varepsilon^2}  \int_\Omega L \left( u(x) \right) \ dx. 
\end{equation}
Together, the first and third terms of $F_{\alpha,\varepsilon}$ define a prior; the image is assumed to take values in $T$ and be smooth. Clearly, for $\varepsilon$ small, minimizing sequences must take values very near $T$. 
When there is no data present, \ie, $\alpha \to \infty$, the energy in \eqref{eq:relax} simplifies to the geometric problem
\begin{equation} \label{e:geo}
\inf_{u \in H^1(\Omega;\mathbb R^k )} \ F_{\infty,\varepsilon}(u), 
\qquad \textrm{where} \  
F_{\infty,\varepsilon}(u) = 
\frac{1}{2}  \int_\Omega  |\nabla u(x) |^2 \ dx+ \ \frac{1}{\varepsilon^2}  \int_\Omega L \left( u(x) \right) \ dx. 
\end{equation}
Energies of this general form, as well as the language ``target set'',  appear, \eg, in \cite{Rubinstein_1989}. 
It is difficult to prove general theorems about when solutions of \eqref{eq:relax} or \eqref{e:geo} converge to solutions of \eqref{eq:original} (if they exist!) as $\varepsilon \searrow 0$ for general target sets, $T$; generally each target set is treated on a case-by-case basis. 

In what follows, we describe in more detail a few choices of target set, $T$, in \eqref{eq:original}, their associated auxiliary functions in  \eqref{eq:relax} and \eqref{e:geo}, along with numerous applications; a summary of various choices of $T$ is given in Table \ref{t:geo}. Our intent is to motivate an alternative relaxation of  \eqref{eq:original}, discussed below, which can be applied to all of these choices of $T$.

\begin{table}[t!]
\begin{center}
\begin{tabular}{lllll}
$k$ & $T$ & L(x) & comment & section \\
\hline
$k$ & $\mathbb R^k$ & $0$ & harmonic function & \\
$k$ & $T \subset \mathbb R^k$ convex & $\frac{1}{2}\textrm{dist}^2(x, T)$ & convex set-valued field & \S\ref{s:Peppers} \\
$n^2$ & $\textrm{SPD}(n)$ & $\frac{1}{2}\textrm{dist}^2\left(x, \textrm{SPD}(n)\right)$ & $\textrm{SPD}$ matrix-valued field &  \S\ref{s:SPD}, \S\ref{s:MRI} \\
\hline
1 & $\{\pm 1 \}$ & $\frac{1}{4} (x^2-1)^2$ &  Allen-Cahn  & \\
2 & $\mathbb S^1$ & $\frac{1}{4} ( |x|^2-1)^2$ & Ginzburg-Landau  &  \S\ref{s:Peppers} \\
k & $\mathbb S^{k-1}$ & $\frac{1}{4} ( |x|^2-1)^2$ & sphere-valued field & \S\ref{s:lemniscate}, \S\ref{s:S2}  \\
$n^2$ & $O(n)$ & $\frac{1}{4} \| x^t x - I_n \|^2_F $ &   orthogonal matrix-valued field  & \\
\hline
$k$ & coordinate axes, $\Sigma_k$ &  $\frac{1}{4} \sum_{i \neq j } x_i^2 x_j^2$ & Dirichlet partitions & \\
&$ \mathbb{RP}^1$& & line field & \S\ref{s:rpk}, \S\ref{s:finger}
\end{tabular}
\end{center}
\caption{Examples of target sets, $T$, and penalization functions, $L$. We've grouped the examples by convex sets (top), subsets of the Euclidean sphere (middle), and other (bottom).}
\label{t:geo}
\end{table}

\subsection*{Target-valued maps and applications in imaging, inverse problems, and geometry} 

\subsubsection*{Subsets of the Euclidean sphere}
For $T=\{\pm 1 \} = \mathbb S^0 \subset \mathbb R^1$,  the auxiliary function $L(x) = \frac{1}{4} (x^2-1)^2$ can be used. The energy, $F_{\infty,\varepsilon}$ in \eqref{e:geo}, corresponds to the Allen-Cahn equation \cite{Allen_1979}. 
Modica and Mortola showed that a minimizing sequence  $(u_\varepsilon)$ of $\varepsilon F_{\infty,\varepsilon}$ converges (along a subsequence) to $\chi_D - \chi_{\Omega \setminus D}$ in $L^1$ for some $D \subset \Omega$. 
Furthermore,   
$\varepsilon F_{\infty,\varepsilon}(u_\varepsilon) \to \frac{2\sqrt{2}}{3} \mathcal{H}^{d-1}(\partial D)$ as $\varepsilon \to 0$ \cite{modica1977esempio}. For small $\varepsilon>0$, the gradient flow of $ F_{\infty,\varepsilon}$ approximates mean curvature flow. This energy serves as a building block for a variety of pattern formation models. 
When a data fidelity term is added, as in \eqref{eq:relax}, we obtain the Rudin-Osher-Fatemi (ROF) functional, 
$$  \int_\Omega |\nabla u | + \frac{1}{2 \alpha} \int_\Omega (u - f)^2.
$$
When used for image denoising, the total variation term serves as an `edge preserving regularizer'  \cite{Rudin_1992}.

For $T = \mathbb S^1 \subset \mathbb R^2$, the auxiliary function $L(x) = \frac{1}{4} (|x|^2-1)^2$ can be used. The energy in \eqref{e:geo} corresponds to the  Ginzburg-Landau equation \cite{bethuel_ginzburg-landau_1994} and has applications in  superconductors and superfluids. 
In imaging, the target set $T = \mathbb S^1$ naturally arises when one tries to recover spatially dependent phase information, such as in Interferometric Synthetic Aperture Radar (InSAR) \cite{Rocca1997,Kampes2006}. Here, the phase difference between an interferometric pair of SAR images, obtained from slightly different camera angles, can be used to construct very accurate elevation maps. 
The solution of \eqref{e:geo} with additional boundary conditions imposed can also be used to design $d=2$-dimensional cross fields, which have applications in, \eg, computer graphics and quad mesh generation \cite{Viertel2017}.  
Finally, this problem is related to the simplification of vector fields for visualization  \cite{Skraba_2015}. 

For $T = \mathbb S^{2} \subset \mathbb R^3$,  the gradient flows of these energies are related to the heat flow of harmonic maps to $\mathbb S^2$ \cite{weinan2001numerical}. 
These equations can be obtained as simplifications of the Landau-Lifschitz equation describing  non-equilibrium magnetism. 
This is also a simplification of the energy (the one-constant approximation) appearing in the Oseen-Frank theory for liquid crystals, where the field represents the preferred direction of molecular alignment \cite{Majumdar_2009,Ball_2017}.

The problem in \eqref{e:geo} with $T = O(n) \subset \{ x \in \mathbb R^{n\times n} \colon \|  x \|_F = 1 \}$ was recently studied by the authors in \cite{OnValuedFields}. Here, it is natural to associate the auxiliary function $L(x) = \frac{1}{4} \| x^t x - I_n \|^2_F $, where $\|\cdot \|_F$ denotes the Frobenius norm. Since $O(1) \cong S^0$, this energy reduces to the Allen-Cahn energy when $n=1$. Recalling that  $O(n) = SO(n) \sqcup SO^-(n)$ and $SO(2) \cong S^1$, for  $n=2$, the gradient flow of this energy with initial conditions taken in $SO(2)$ reduces to the Ginzburg-Landau gradient flow  \cite{OnValuedFields}. This energy can be considered as a model problem for crystallography, where one considers a field that takes values in $SO(3)$ modulo the symmetry group of the crystal. 
This is also a model problem for the three-dimensional cross field design problem \cite{Viertel2017}. 
Finally, this problem is related to problems in rigid motion planning, where one tries to find a time-dependent trajectory, $u \colon [t_1,t_2] \to T$, where the function takes values in a set that describes admissible rigid motions, such as, \eg, $T=SO(3)$ \cite{Sciavicco_2000}.

\subsubsection*{Convex sets}
When the target set $T$ is a convex set of $\mathbb R^k$, we can generally take the auxiliary function to be $L(x) = \frac{1}{2} \textrm{dist}^2(x, T) = \frac{1}{2} \min_{y \in T} \ \textrm{dist}^2(x, y)$. For RGB images, the image takes values in the cube, $T=[0,1]^3 \subset \mathbb R^3$, as further discussed in Section \ref{s:Peppers}.

For $T = \textrm{SPD}(n)$, the set of semi-positive definite (SPD) matrices, the inverse problem in \eqref{eq:relax} appears in diffusion tensor magnetic resonance imaging (MRI) \cite{Wandell_2016,Lenglet_2009}. This application will be further discussed in Section \ref{s:MRI}. 
 
\subsubsection*{Other target sets}
For the coordinate axis, $T = \Sigma_k := \{ x \in \mathbb R^k \colon \sum_{i\neq j} x_i^2 x_j^2 = 0\}$,  the minimizer of \eqref{e:geo} with an additional norm constraint gives Dirichlet partitions of $\Omega$ in the limit as $\varepsilon \to 0$; see \cite{caffarelli2007optimal,MBOforDirPart2018}. Recently, Dirichlet partitions have been used for image segmentation and data clustering \cite{Partition2013,Zosso2015,Reeb2016}. 

The target set, $T = \mathbb{RP}^n$ is related to the Landau-de Gennes model, where the field describing the local orientation of a crystal is described by a $Q$-tensor \cite{Majumdar_2009,Ball_2017}. While this theory was originally used to describe nematic liquid crystals, it has also been used to describe the orientations of RNA and carbon nanotubes.
Thinking of real projective space as the quotient space obtained from the $n$-sphere after identifying antipodal points, fields with values in $T=\mathbb{RP}^1$ are referred to as  \emph{line fields}, where a pair of antipodal directions is assigned to each point.   This application will be further discussed in Section \ref{s:finger}. 

\bigskip

Further discussion of inverse problems for manifold-valued images can be found in  \cite{Weinmann_2014,Bacak_2016,Grohs_2016,Laus_2017}. 

\begin{algorithm}[t!]
\DontPrintSemicolon
 \KwIn{Let $\tau, \lambda> 0$. Set $\Omega \in \mathbb{R}^d$, the target space as $T \in \mathbb{R}^k$, the data as $f \in L^\infty(\Omega;\mathbb R^k)$ and the initial guess as $u_0 = f$.}
 \KwOut{ A function $u \in L^\infty(\Omega; T) $ that approximately minimizes   \eqref{e:relax2}.}
 Set $s=1$.\;
 \While{not converged}{
{\bf 1.  Diffusion Step.} Extend $u_{s-1}$ and $f$ to $\mathbb R^d \setminus \Omega$ by zero. 
Solve the initial value problem for the free space diffusion equation until time $\tau$:
\begin{align*}
&\partial_t v(t,x) = \Delta v (t,x) \\
&v(0,x) = (1-\lambda) u_{s-1}(x)+ \lambda f(x).
\end{align*}
Let $\tilde u(x) = v(\tau,x)$\; 
{\bf 2. Projection Step.} Define $u_s \in L^\infty(\Omega; T)$ by point-wise applying $\Pi_T$ to  $\tilde u$,
$$
u_s(x) = \Pi_T \tilde u(x) \qquad \qquad x \in \Omega.
$$
Set $s = s+1$.\;
 }
\caption{A diffusion generated method for approximating minimizers of the energy in  \eqref{e:relax2}. } 
\label{a:alg1}
\end{algorithm}

\subsection*{Results}
In this paper, we derive and study an alternative relaxation of \eqref{eq:original} than  \eqref{eq:relax} based on the heat kernel. Namely, for $\tau > 0$, $\lambda \in (0,1)$, and $f \in L^\infty(\Omega; \mathbb R^k)$, we consider the  relaxation of \eqref{eq:original} given by 
\begin{subequations} \label{e:relax2}
\begin{equation} \label{eq:relaxed}
\inf_{u \in L^2(\Omega; T)} \ E_{\lambda,\tau}(u),
\end{equation}
where
\begin{equation} \label{e:EnergyLambda}
E_{\lambda,\tau}(u) =  - \frac{1}{2} \langle u, (e^{\Delta \tau} - I) u \rangle + \frac{\lambda}{2} \langle u- f, e^{\Delta \tau} (u-f) \rangle. 
\end{equation}
\end{subequations}
Here, $\langle \cdot, \cdot \rangle$, is the $L^2(\Omega;\mathbb R^k)$-inner product and $e^{\Delta \tau}$ denotes the solution operator for the free space diffusion equation at time $\tau$. If $u$ is a vector-valued field, the diffusion operator is understood to be applied component-wise. 
Here, the parameter $\tau$ measures the relaxation of the problem and the parameter $\lambda = \frac{\tau}{\alpha}$ controls the data fidelity. 
The two terms in \eqref{e:EnergyLambda} come from relaxing the two terms in the energy in \eqref{eq:original}.  This is explained in Section~\ref{s:RelaxProb}, together with conditions on the target set $T$ such that \eqref{e:relax2} is well-defined, and interpretations of \eqref{e:relax2}. 
The second term of $E_{\lambda,\tau}$ in \eqref{e:EnergyLambda} is similar to the region-based active contour model in \cite{li2008minimization}, where the idea is to use a nonlocal fidelity term to characterize the image better. 

The main contribution of this paper is to derive and analyze a diffusion generated method to solve \eqref{e:relax2} for a wide class of target sets, $T$, including those discussed in Table~\ref{t:geo}. 
The proposed algorithm is given in Algorithm~\ref{a:alg1}. 
The Algorithm consists of taking a convex combination of the previous time step and the data, diffusing until time $\tau$, and applying a map, $\Pi_T$, point-wise to the resulting function. Here $\Pi_T$ is the convex projection onto the target set, $T$, in the case that $T$ is convex and the closest point mapping otherwise.  A derivation of Algorithm~\ref{a:alg1} for the target sets considered, as well as convergence properties of the algorithm are given in Section~\ref{s:alg}. 

Algorithm~\ref{a:alg1}  is conceptually simple, computationally efficient, easy to implement, and applicable to a broad class of problems.  
Algorithm~\ref{a:alg1} can be interpreted as a splitting method for  \eqref{eq:relax}; see Section \ref{s:EnergySplit}. However, we prefer to interpret 
Algorithm~\ref{a:alg1} in terms of \eqref{e:relax2} since neither rely on the auxiliary function $L$  as in  \eqref{eq:relax}.  There are two extremes for Algorithm~\ref{a:alg1}. If $\lambda = 0$, we ignore the data and find an approximate harmonic function with values in $T$. This is similar to the geometric problem \eqref{e:geo} discussed above. If $\lambda = 1$, we simply dampen the highly oscillatory terms in $f$ and apply the mapping $\Pi_T$ once. 

\begin{rem}
The argument in \cite[Section 4.1]{LauxYip2018} shows that the output of Algorithm \ref{a:alg1} satisfies homogeneous Neumann boundary conditions on $\partial \Omega$. Dirichlet boundary conditions are also discussed by Laux and Yip, but we don't impose these conditions in the present work. 
\end{rem}

In Section~\ref{s:CompEx}, we use the proposed method to study a variety of synthetic and test numerical experiments. 
The target sets considered include 
convex sets (\S\ref{s:Peppers}),
$\textrm{SPD}(3)$ (\S\ref{s:SPD}, \S\ref{s:MRI}), 
$\mathbb S^1$ (\S\ref{s:Peppers}), 
$\mathbb S^2$ (\S\ref{s:lemniscate}, \S\ref{s:S2}),  and
$ \mathbb{RP}^1$ (\S\ref{s:rpk}, \S\ref{s:finger}), 
which have applications as described above. 

\subsection*{Previous work on diffusion generated methods}
The proposed method (Algorithm~\ref{a:alg1}) falls into the class of  diffusion generated methods (DGMs). 
DGMs were first introduced for $T = \{\pm 1 \}$ and showed to be associated with mean curvature flow  in \cite{MBO1993,merriman1992diffusion}. 
DGMs have also been generalized to generate high order geometric motions, such as Wilmore and surface diffusion flows, in \cite{esedoglu2008threshold}. 
In \cite{esedoglu2010diffusion}, the authors used the diffusion of the distance function to generate mean curvature flow where the thresholding step was replaced by redistancing. 
DGMs were recently shown to be stable and generalized to multiphase mean curvature flow in \cite{esedoglu2015threshold} and applied to wetting problems in \cite{xu2016efficient}. 
DGMs have been used for inverse problems for $T = \{\pm 1 \}$ in \cite{wang2017efficient,esedog2006threshold}. The convergence rate of a DGM to a stationary point was proven in \cite{OnValuedFields}. DGMs for $T = \mathbb S^1$ were introduced in  \cite{Ruuth_2001}, used for quad mesh generation  in \cite{Viertel2017}, and proven to be  convergent  in \cite{LauxYip2018}. DGMs for $T= \mathbb S^2$ were also studied in \cite{weinan2001numerical}. 
Finally, DGMs for $T= O(n)$ was introduced and studied in  \cite{OnValuedFields}. 

\subsection*{Outline}
In Section~\ref{s:RelaxProb}, we describe properties of \eqref{e:relax2}. 
In Section~\ref{s:alg}, we derive and study Algorithm~\ref{a:alg1} for the two types of target sets considered. 
In Section~\ref{s:CompEx}, we use the proposed methods to study a variety of numerical experiments. 
We conclude in Section~\ref{s:Disc} with a discussion.

\section{Properties and interpretation of the relaxed problem, \eqref{e:relax2}} \label{s:RelaxProb}
In this section, we motivate and derive properties of the energy, $E_{\lambda,\tau}$, in \eqref{e:relax2}, show the existence of solutions to \eqref{e:relax2}, and give two interpretations of $E_{\lambda,\tau}$. 

\subsection{Motivation and properties of $E_{\lambda,\tau}$}
For $u_0 \in L^1(\Omega; \mathbb R^k)$, 
we write $e^{\Delta \tau} u_0$ to denote the solution to the free space heat equation with initial condition $u_0$ at time $\tau$,
$$
\Big(e^{\Delta \tau} u_0\Big) (x) = u(x,\tau) = (4 \pi \tau )^{-d/2} \int_{\Omega} e^{-|x - y |^2/4\tau} u_0(y) \ dy, \qquad \qquad x \in \Omega. 
$$ 
Let $\langle u, v \rangle = \int_{\Omega} u(x)\cdot v(x) \ dx $ denote the $L^2(\Omega;\mathbb R^k)$ inner product. Here $\cdot$ is interpreted as the dot product in $\mathbb R^k$ or the Frobenius inner product if $u$ and $v$ are matrix valued fields.

For $\tau \sim o(1)$, we write
\begin{equation} \label{eq:DirHeat}
\| \nabla u \|^2 = \langle u, -\Delta u \rangle \approx - \frac{1}{\tau} \langle u, (e^{\Delta \tau} - I) u \rangle. 
\end{equation}
Also, we have
\begin{equation} \label{eq:regionbased}
\| u-f \|^2 = \langle u-f,u-f \rangle \approx \langle u-f, e^{\Delta \tau} (u-f) \rangle.
\end{equation}
Then, using \eqref{eq:DirHeat} and \eqref{eq:regionbased}, we approximate the energy in the inverse problem \eqref{eq:original} by
$$ 
E_\alpha(u) =  -\frac{1}{2\tau}  \langle u, (e^{\Delta \tau} - I) u \rangle  + \frac{1}{2 \alpha}  \langle u- f, e^{\Delta \tau} (u-f) \rangle  \ \ +  \ \ O(\tau). 
$$
Defining $\lambda = \frac \tau \alpha$, we obtain 
$$
E_\alpha(u) =  \tau^{-1}  E_{\lambda, \tau}(u)  \ +  \ O(\tau), 
$$
where $E_{\lambda, \tau}$ is given in  \eqref{e:EnergyLambda}.  

\begin{lem} \label{lem:function}
Assume $\lambda \in (0,1)$, $\tau > 0$, and $f \in L^\infty(\Omega;\mathbb R^k)$. Then the following properties hold for the functional $E_{\lambda,\tau}$ defined in \eqref{e:EnergyLambda}. 
\begin{enumerate}
\item[(i)] $E_{\lambda,\tau}$ is non-negative on $L^2(\Omega; \mathbb R^k)$. 
\item[(ii)] $E_{\lambda,\tau}$ is continuous with respect to the strong topology on $L^2(\Omega;\mathbb R^k)$.
\item[(iii)] The Fr\'echet derivative of $E_{\lambda,\tau} \colon L^2(\Omega; \mathbb R^k) \to \mathbb R$ with respect to $u$ is 
\[\frac{\delta E_{\lambda,\tau}(u)}{\delta u} = - \left( e^{\Delta \tau} - I \right) u  + \lambda e^{\Delta \tau} (u-f) . \]
\item[(iv)] The first variation, $\frac{\delta E_{\lambda,\tau}}{\delta u} \colon L^2(\Omega; \mathbb R^k) \to L^2(\Omega; \mathbb R^k)$ is Lipschitz continuous with Lipschitz constant $L = 2- \lambda$. 
\item[(v)] $E_{\lambda,\tau}$ is strongly convex on $L^2(\Omega; \mathbb R^k)$ with constant $\lambda$.
\item[(vi)] We have the bound
$$ 
E_{\lambda,\tau}(u) \geq \frac{\lambda }{2} \left( \| u \| - \left\| e^{\Delta \tau} f \right\| \right)^2 + \frac{\lambda}{2} \left( \left\| e^{\Delta \tau /2} f \right\|^2 -  \left\| e^{\Delta \tau} f \right\|^2 \right),
$$
so the sublevel sets of $E_{\lambda,\tau}$ are bounded in $L^2(\Omega; \mathbb R^k)$. 
\end{enumerate}
\end{lem}
\begin{proof}
(i) We compute
\begin{align*}
E_{\lambda,\tau}(u) &= - \frac{1}{2} \langle u, (e^{\Delta \tau} - I) u \rangle + \frac{\lambda}{2} \langle u- f, e^{\Delta \tau} (u-f) \rangle \\
&= \frac{1}{2} \| u \|^2 - \frac{1}{2} \| e^{\Delta \tau / 2} u \|^2 + \frac{\lambda}{2} \| e^{\Delta \tau/2} (u-f) \|^2 \\
&\geq 0
\end{align*}
(ii) Let $u, v \in L^2(\Omega;\mathbb R^k)$. 
Observe that 
$$
E_{\lambda,\tau}(u) = \frac{1}{2} \left\langle u, \left( I - (1-\lambda) e^{\Delta \tau} \right) u \right\rangle - \lambda \left\langle e^{\Delta \tau} f, u \right \rangle + \frac{\lambda}{2} \left\langle f, e^{\Delta \tau} f \right \rangle. 
$$
Using that for $u_0 \in L^{2}(\Omega; \mathbb R^k)$, we have that  $\| e^{\Delta \tau} u_0\|_{L^2(\Omega; \mathbb R^k)} \leq  \| u_0\|_{L^2(\Omega; \mathbb R^k)}$, we compute
\begin{align*}
| E_{\lambda,\tau}(u)  - E_{\lambda,\tau}(v) | 
& = \left| \frac{1}{2} \left\langle  \left( I - (1-\lambda) e^{\Delta \tau} \right) (u+v),   u-v \right\rangle - \lambda  \left\langle e^{\Delta \tau} f, u-v \right \rangle  \right| \\
& \leq  \left(  \| u+v \| + \lambda \left\| e^{\Delta \tau} f \right\| \right) \| u-v \|. 
\end{align*}
Let $\varepsilon > 0$, $\delta = \min \left\{ 1,  \frac{\varepsilon}{1 + 2 \|v \| + \lambda \left\| e^{\Delta \tau} f \right\|} \right\}$, and $\| u-v \| \leq \delta$. Then 
$$
| E_{\lambda,\tau}(u)  - E_{\lambda,\tau}(v) | 
\leq \left(  \| u-v \| + 2\|v\| + \lambda \left\| e^{\Delta \tau} f \right\| \right) \| u-v \|
\leq \varepsilon.
$$
(iii) From the definition, we directly compute 
\[\delta E_{\lambda,\tau}(u) = \left.\frac{d (E_{\lambda,\tau}(u+s \delta u)) }{ds}\right|_{s=0} =\left\langle \delta u, - \left( e^{\Delta \tau} - I \right) u  + \lambda e^{\Delta \tau} (u-f) \right\rangle.
\]  
(iv) For $\lambda \in (0,1)$, we compute
\begin{align*} 
\left\| \frac{\delta E_{\lambda,\tau}(u)}{\delta u}  - \frac{\delta E_{\lambda,\tau}(u)}{\delta v}  \right\| 
& = \| - \left( e^{\Delta \tau} - I \right) (u-v)  + \lambda e^{\Delta \tau} (u-v) \| \\ 
& \leq (1-\lambda) \| e^{\Delta \tau}(u-v)  \| + \| u-v\| \\ 
& \leq (1-\lambda)  \| u-v \| +  \| u - v \| \\
&= L\| u - v\|.
\end{align*} 
(v) From direct calculation, we have for any $u,v\in L^{\infty}(\Omega;T)$ and any $\gamma \in (0,1) $, 
\begin{align*}
\gamma E_{\lambda,\tau}(u) +  (1-\gamma) E_{\lambda,\tau}(v) - E_{\lambda,\tau}(\gamma u + (1-\gamma) v)  
&= - \frac{1}{2} \gamma(1-\gamma) \left\langle u-v, \left( (1-\lambda) e^{\Delta \tau} - I \right)(u-v) \right\rangle \\ 
&\geq - \frac{1}{2} \gamma(1-\gamma) \left( 1 - \lambda - 1\right) \| u - v \|^2  \\ 
& = \frac{\lambda}{2} \gamma(1-\gamma) \|u-v\|^2 .
\end{align*}
This implies that  $E_{\lambda,\tau}$ is strongly convex with constant $\lambda$. \\
(vi) We compute
\begin{align*}
E_{\lambda,\tau}(u) &= - \frac{1}{2} \left\langle u, (e^{\Delta \tau} - I) u \right\rangle + \frac{\lambda}{2} \left\langle u- f, e^{\Delta \tau} (u-f) \right\rangle \\
&= \frac{1}{2} (1-\lambda) \left\langle u,  ( I - e^{\Delta \tau}) u \right\rangle + \frac{\lambda}{2} \| u\|^2 - \lambda \left\langle e^{\Delta \tau} f, u  \right\rangle 
+ \frac{\lambda}{2} \left\| e^{\Delta \tau / 2} f \right\|^2 \\
&\geq \frac{\lambda}{2} \| u\|^2 - \lambda \left\| e^{\Delta \tau} f \right\| \| u  \|  + \frac{\lambda}{2} \left\| e^{\Delta \tau / 2} f \right\|^2 \\
&= \frac{\lambda }{2} \left( \| u \| - \left\| e^{\Delta \tau} f \right\| \right)^2 + \frac{\lambda}{2} \left( \left\| e^{\Delta \tau /2} f \right\|^2 -  \left\| e^{\Delta \tau} f \right\|^2 \right).
\end{align*}
\end{proof}

\subsection{Existence}
We define $L^p(\Omega; T)$ to be the subset of $L^p(\Omega; \mathbb R^k)$ consisting of maps $\Omega \to \mathbb R^k$ with image essentially in $T \subset \mathbb R^k$. 

\begin{thm} \label{prop:Existence}
Let $T \subset \mathbb R^k$ be a closed subset and assume $\lambda \in (0,1)$, $\tau > 0$, and $f \in L^\infty(\Omega;\mathbb R^k)$. Then there exists a solution $u^\star \in L^2(\Omega; T)$ that attains the infimum in \eqref{e:relax2}. 
\end{thm}
\begin{proof}
We use the direct method in the calculus of variations to establish existence. By Lemma \ref{lem:function}(i), the functional is non-negative on $ L^2(\Omega; T)$. Let $(u_j)_{j \in \mathbb N} \subset L^\infty(\Omega; T)$ be a minimizing sequence, \ie, $E_{\lambda,\tau}(u_j) \to E_{\lambda,\tau}^\star = \inf_{u \in L^2(\Omega;T)} E_{\lambda,\tau}(u)$ as $j\to \infty$. 
By Lemma \ref{lem:function}(vi), the minimizing sequence is bounded in $L^{2}(\Omega; T)$. 
By, \eg, \cite[Thm 1.1.2]{Evans_1990}, there exists a subsequence, which we continue to denote by $(u_j)_{j \in \mathbb N}$ and $u^\star \in L^{2}(\Omega; T)$, such that $ u_j \wto u^\star$. By the continuity of $E_{\lambda,\tau}$ with respect to the weak $L^2(\Omega; T)$ topology (see Lemma \ref{lem:function}(ii)), 
$$
E_{\lambda,\tau}^\star = \lim_{j \to \infty} E_{\lambda,\tau}(u_j) = E_{\lambda,\tau}( u^\star),
$$
which shows that $u^\star$ attains the infimum. 
\end{proof}

\subsection{Fourier interpretation of the relaxed energy}
In this section, we use the Fourier transform to describe the sense in which the minimizer in \eqref{e:relax2} achieves a balance between smoothness and fidelity to the data, $f$.  
Recall that the solution to the diffusion equation with initial data, $u(x,t) = u_0(x)$ can be expressed using the Fourier transform,
$ u(x,t) = (2 \pi)^{-d/2} \int_{\mathbb R^d} e^{- \tau |k|^2 } \hat u_0(k) e^{ik\cdot x} \ dk$, where $\hat u_0(k) = (2 \pi)^{-d/2} \int_{\mathbb R^d} u_0(y) e^{- i k\cdot y} \ dy$. It follows that the energy in \eqref{e:EnergyLambda} can be written
$$
E_{\lambda,\tau}(u) =   \int_{\mathbb R^d} \ \frac{1}{2}   \left(1-e^{-\tau |k|^2}\right) |  \hat u(k)  |^2 + \frac{\lambda}{2} e^{- |k|^2 \tau}   |  \hat u_0(k) - \hat f(k) |^2 \ dk. 
$$
This transformation shows that any test function $u$ with small energy has the following two properties:
\begin{itemize}
\item The Fourier transform of $u$, given by $\hat u(k) $, should be small for $|k| \gg \tau^{-1/2}$. 
\item The Fourier transform of the residual, $u-f$, should be small for $|k| \ll \tau^{-1/2}$. 
\end{itemize}

\subsection{Perimeter interpretation of the relaxed energy}
In this section, for convenience, we discuss the target set $T = \{0,1\}$, which can be obtained by shifting and rescaling  $T= \{\pm 1\}$. 
In \cite{esedoglu2015threshold}, Esedoglu and Otto use $k$ indicator functions $\{u_i\}_{i=1}^k$ of $k$ domains $\{D_i\}_{i=1}^k$ to implicitly represent each domain and the interfaces, $\gamma_{ij}$, between $D_i$ and $D_j$. When $\tau \ll 1$, the  area of $\gamma_{ij}$ can be approximated by 
\begin{equation}\label{e:3.1}
|\gamma_{ij}|\approx \frac{1}{\sqrt{\tau}}\int u_1 G_{\tau}*u_2 \ dx,
\qquad \qquad \textrm{where} \ \ 
G_{\tau}( x)=\frac{1}{(4\pi \tau)^{d/2}} \exp\left(-\frac{| x|^2}{4\tau}\right)
\end{equation}
is the Gaussian kernel; see also \cite{alberti1998non,miranda2007short}. Up  to a constant, this is equivalent to the regularity term in \eqref{eq:DirHeat}.
The expression in \eqref{e:3.1} was shown to $\Gamma$-converge to the area of $\gamma_{ij}$ when $\tau \searrow 0$ in \cite{alberti1998non,miranda2007short,esedoglu2015threshold}.  In \cite{esedoglu2015threshold}, based on this approximation, Esedoglu and Otto successfully generalized the original MBO method to a general threshold dynamics method to model the multiphase mean curvature flow by using a relaxation and linearization procedure. This procedure provides a proof of unconditional stability and consistency of the algorithm. In \cite{laux2016convergence}, the algorithm was rigorously proved to converge to multiphase mean curvature flow with an angle constraint at the multiple junction when $\tau \searrow 0$. A convergence proof for $T=\mathbb S^1$ is given in \cite{laux2017convergence}.

\section{Derivation and properties of the diffusion generated algorithm} \label{s:alg}

In the following two subsections, we separately derive a diffusion generated method when the target set is (i) a convex set or (ii) a closed subset of the unit sphere. Both derivations lead to Algorithm \ref{a:alg1}. In Section~\ref{s:EnergySplit}, we give an energy splitting interpretation of Algorithm~\ref{a:alg1}.

\subsection{The target set is a closed convex set} \label{s:Tconvex}
When the target set, $T$, is convex, we directly use the gradient projection algorithm (see, \eg, \cite{Bertsekas2015}) for a fixed time step size, $1$, to find the solution of \eqref{e:relax2}. That is, if  $u_n$ is the solution at the $n$-th iteration, we define the $n+1$-th iteration, $u_{n+1}$, by 
\[u_{n+1} = \tilde\Pi_T\left(u_n - \frac{\delta E_{\lambda,\tau}(u_n)}{\delta u}\right), \]
where $\tilde\Pi_T(v)$ is the convex projection of $v\in L^2(\Omega; \mathbb{R}^k)$ into $L^2(\Omega; T)$, i.e., 
$$
\tilde\Pi_T(v) = \arg\min_{u \in L^2(\Omega; T)} \| u - v \|^2.
$$
Here, 
\[\frac{\delta E_{\lambda,\tau}(u_n)}{\delta u} = -   \left( e^{\Delta \tau} - I \right) u_n  + \lambda e^{\Delta \tau} (u_n-f) \]
is the variation of $E_{\lambda,\tau}(u)$ with respect to $u$ at $u=u_n$; see Lemma \ref{lem:function}(iii).
Direct calculation gives
$$
u_{n+1} = \tilde\Pi_T \left( e^{\Delta \tau} ((1-\lambda) u_n +\lambda f)\right). 
$$
Here, we take a convex combination of the data and current iterate, solve the diffusion equation until time $\tau$, and project into the target set, $T$.
Since $e^{\Delta \tau} ((1-\lambda) u_n +\lambda f$ is a $C^\infty(\Omega; \mathbb R^k)$ function, the projection step is equivalent to 
\begin{equation}\label{eq:iteration}
u_{n+1} = \Pi_T \left( e^{\Delta \tau} ((1-\lambda) u_n +\lambda f)\right),
\end{equation}
where $\Pi_T(v)$ is the point-wise convex projection of $v\in \mathbb{R}^k$ into the target set $T$. 
The algorithm is summarized in Algorithm \ref{a:alg1}. 

The following theorem gives a convergence result for Algorithm \ref{a:alg1} for $T$ convex. 
\begin{thm}
Let $T$ be a closed and convex set and assume $\tau > 0$, $\lambda \in (0,1)$, and $f \in L^\infty(\Omega; \mathbb R^k)$. The sequence generated by \eqref{eq:iteration} for any initial condition $u_0 \in L^2(\Omega; \mathbb R^k)$
strongly converges in $L^2(\Omega; T)$ to the unique minimum of \eqref{e:relax2}, \ie, $\|u_n - u^\star\| \to 0$. Furthermore, the sequence converges at a geometric rate $1-\lambda$, \ie,
\[\|u_{n} - u^\star \| \leq (1-\lambda)^n \|u_{0} - u^\star \|,  \qquad \qquad  \forall n \in \mathbb{N}. \]
\end{thm}
\begin{proof}
By Theorem~\ref{prop:Existence}, there exists an optimal solution $u^\star \in L^2(\Omega; T)$ to \eqref{e:relax2}, necessarily satisfying
$$
\left \langle \left. \frac{\delta E_{\lambda,\tau}(u)}{\delta u}\right|_{u=u^\star}, v - u^\star \right \rangle \geq 0, \qquad \qquad \forall v \in L^2(\Omega; T).
$$
The uniqueness of this solution follows from the convexity of $L^2(\Omega; T)$ and strong convexity of $E_{\lambda,\tau}$; see Lemma \ref{lem:function}(v). 
Adding and subtracting $u^\star$, we can rewrite this as
$$
\left \langle u^\star - \left( u^\star -  \left. \frac{\delta E_{\lambda,\tau}(u)}{\delta u}\right|_{u=u^\star} \right), v - u^\star \right \rangle \geq 0, \qquad \qquad \forall v \in L^2(\Omega; T).
$$
By \cite[p.40]{Ekeland_1999}, we conclude that
\[u^\star = \Pi_T\left(u^\star- \left. \frac{\delta E_{\lambda,\tau}(u)}{\delta u}\right|_{u=u^\star}\right), 
\qquad \qquad \textrm{a.e.} \ \ x\in \Omega.
\]

Using the definition of the method in \eqref{eq:iteration}, we obtain 
\[\|u_{n+1}-u^\star\|^2 = \left\|\Pi_T\left(u_{n}-\left.\frac{\delta E_{\lambda,\tau}(u)}{\delta u}\right|_{u=u_{n}} \right)- \Pi_T\left(u^\star-\left.\frac{\delta E_{\lambda,\tau}(u)}{\delta u}\right|_{u=u^\star}\right) \right\|^2. 
\]
By the non-expansiveness of the convex projection, it follows that
{\footnotesize
\begin{align} 
\label{ineq:convproof1} 
&  \|u_{n+1}-u^\star\|^2  \\
& \leq \left\|u_{n}-\left.\frac{\delta E_{\lambda,\tau}(u)}{\delta u}\right|_{u=u_{n}} - u^\star+\left.\frac{\delta E_{\lambda,\tau}(u)}{\delta u}\right|_{u=u^\star} \right\|^2 \nonumber \\
& =  \left\|u_{n}-u^\star \right\|^2 - 2 \left\langle u_n-u^\star, \left.\frac{\delta E_{\lambda,\tau}(u)}{\delta u}\right|_{u=u_{n}} - \left.\frac{\delta E_{\lambda,\tau}(u)}{\delta u}\right|_{u=u^\star}\right\rangle 
+ \left \| \left.\frac{\delta E_{\lambda,\tau}(u)}{\delta u}\right|_{u=u_{n}} - \left.\frac{\delta E_{\lambda,\tau}(u)}{\delta u}\right|_{u=u^\star} \right\|^2  \nonumber
\end{align}
}

Now adding the two inequalities 
\begin{align*}
E_{\lambda, \tau}(u) &\geq E_{\lambda, \tau}(u^\star) + \left\langle \left.\frac{\delta E_{\lambda,\tau}(u)}{\delta u}\right|_{u=u^\star}, u-u^\star \right \rangle
+ \frac{L}{2} \| u-u^\star\|^2 \\
E_{\lambda, \tau}(u^\star) &\geq E_{\lambda, \tau}(u) + \left\langle \left.\frac{\delta E_{\lambda,\tau}(u)}{\delta u}\right|_{u=u}, u^\star-u \right \rangle
+ \frac{L}{2} \| u-u^\star\|^2
\end{align*} 
we obtain 
$$
\left\langle \left.\frac{\delta E_{\lambda,\tau}(u)}{\delta u}\right|_{u=u} - \left.\frac{\delta E_{\lambda,\tau}(u)}{\delta u}\right|_{u=u^\star} , u-u^\star \right \rangle
\geq L \| u - u^\star \|^2. 
$$
Combining this with the inequality 
$$
\left \| \left.\frac{\delta E_{\lambda,\tau}(u)}{\delta u}\right|_{u=u} - \left.\frac{\delta E_{\lambda,\tau}(u)}{\delta u}\right|_{u=u^\star} \right\|
\leq L \| u - u^\star\|, 
$$
where $L$ is the Lipschitz constant computed in Lemma \ref{lem:function}(iv), 
we obtain the inequality
$$
\left\langle \left.\frac{\delta E_{\lambda,\tau}(u)}{\delta u}\right|_{u=u} - \left.\frac{\delta E_{\lambda,\tau}(u)}{\delta u}\right|_{u=u^\star} , u-u^\star \right \rangle
\geq 
\frac{1}{L} \left \| \left.\frac{\delta E_{\lambda,\tau}(u)}{\delta u}\right|_{u=u} - \left.\frac{\delta E_{\lambda,\tau}(u)}{\delta u}\right|_{u=u^\star} \right\|^2. 
$$
Using this inequality in \eqref{ineq:convproof1}, we obtain 
\begin{equation} \label{ineq:convproof2}
 \left( \frac{2}{L} - 1 \right) \left \| \left.\frac{\delta E_{\lambda,\tau}(u)}{\delta u}\right|_{u=u_{n}} - \left.\frac{\delta E_{\lambda,\tau}(u)}{\delta u}\right|_{u=u^\star} \right\|^2 \leq \left\|u_{n}-u^\star \right\|^2 -\|u_{n+1}-u^\star\|^2.
 \end{equation}
On one hand, by summing \eqref{ineq:convproof2} from $n=0$ to $N$ and letting $N \to \infty$, we obtain
\[ \left( \frac{2}{L} - 1 \right)  \sum_{n=0}^\infty\left \| \left.\frac{\delta E_{\lambda,\tau}(u)}{\delta u}\right|_{u=u_{n}} - \left.\frac{\delta E_{\lambda,\tau}(u)}{\delta u}\right|_{u=u^\star} \right\|^2 \leq \left\|u_{0}-u^\star \right\|^2 -\|u_{\infty}-u^\star\|^2 \leq  \left\|u_{0}-u^\star \right\|^2 < \infty.\]
We conclude that 
\[\lim_{n\to \infty} \left.\frac{\delta E_{\lambda,\tau}(u)}{\delta u}\right|_{u=u_{n}} = \left.\frac{\delta E_{\lambda,\tau}(u)}{\delta u}\right|_{u=u^\star} . \]
On the other hand, from \eqref{ineq:convproof2},  we also have 
\[\|u_{n+1}-u^\star\|^2  \leq  \left\|u_{n}-u^\star \right\|^2 - \left(  \frac{2}{L} -1 \right)  \left \| \left.\frac{\delta E_{\lambda,\tau}(u)}{\delta u}\right|_{u=u_{n}} - \left.\frac{\delta E_{\lambda,\tau}(u)}{\delta u}\right|_{u=u^\star} \right\|^2 .
\]
By summing these relations over $n  = M, . . . , N$ for arbitrary $M$ and $N$ with
$M < N$, taking the $\limsup$ as $N \to \infty$, and taking the $\liminf$ as $M \to \infty$, we obtain 
{\small
\begin{align*}
\limsup_{N\to \infty} \|u_{N+1}-u^\star\|^2 
& \leq \liminf_{M \to \infty}\|u_{M}-u^\star\|^2 -  \left(\frac{2}{L} -1 \right)\liminf_{M \to \infty} \sum_{n=M}^\infty \left \| \left.\frac{\delta E_{\lambda,\tau}(u)}{\delta u}\right|_{u=u_{n}} - \left.\frac{\delta E_{\lambda,\tau}(u)}{\delta u}\right|_{u=u^\star} \right\|^2 \\
& = \liminf_{M \to \infty}\|u_{M}-u^\star\|^2.
\end{align*} 
}
Hence, we are led to the conclusion that the  sequence $\{\|u_n-u^\star\|\}$ is convergent hence bounded, implying that 
$\{\|u_n\|\}$ is also bounded. Thus, $\{ u_n\}_{n \in \mathbb N}$ weakly converges in $L^2(\Omega; T)$ to a $\tilde u \in L^2(\Omega; T)$,  \ie,
$ \lim_{n\to \infty} \langle u_n- \tilde u, v \rangle = 0, \ \forall v \in L^2(\Omega; T)$. To see that $\tilde u = u^\star$, prove strong convergence, and obtain the convergence rate, we use the non-expansiveness of the projection to obtain 
\begin{align*}
\|u_{n+1}-u^\star\| 
& = \left\|\Pi_T\left(u_{n}-\left.\frac{\delta E_{\lambda,\tau}(u)}{\delta u}\right|_{u=u_{n}} \right)- \Pi_T\left(u^\star-\left.\frac{\delta E_{\lambda,\tau}(u)}{\delta u}\right|_{u=u^\star}\right) \right\|  \\
 &\leq \left\|u_{n}-\left.\frac{\delta E_{\lambda,\tau}(u)}{\delta u}\right|_{u=u_{n}} - u^\star+\left.\frac{\delta E_{\lambda,\tau}(u)}{\delta u}\right|_{u=u^\star} \right\| \\
 & = \left\|\left( 1- \lambda \right) e^{\Delta \tau} (u_n-u^\star) \right\| \\
 &\leq \left( 1- \lambda \right)  \|u_n-u^\star\|.  
 \end{align*}
The desired statement then follows. 
\end{proof}

\subsection{The target set is a closed subset of the sphere} \label{s:TargetSetSphere}
We consider a target set, $T$, satisfying the following properties:
\begin{enumerate}
\item $T$ is a closed subset of the sphere,  $\mathbb S^{k-1}$, \ie, $T\subseteq \{ x \colon | x | = 1 \} \subset \mathbb R^k. $
\item There exists a measure zero set, $\mathcal N \subset \mathbb R^k$, such that for every point in $\mathbb R^k \setminus \mathcal N$, we can define the closest point map, $\Pi_T
\colon \mathbb R^k \setminus \mathcal N \to T$, which takes points to their closest point in $T$, 
$$
\Pi_T x = \arg \min_{y \in T}  \ | x - y |^2. 
$$
\item We define the closed convex set $\mathcal B = \textrm{conv}(T)$ to be the convex hull of $T$. 
\end{enumerate}

\subsubsection*{Example} For  $T = \mathbb S^{k-1}$, we define $\mathcal N = \{0\}$, $
\Pi_T x = \frac{x}{|x|}$, and $\mathcal B = \{ x \colon |x | \leq 1 \}$. 

\subsubsection*{Example}For $T = O(n) \subset \mathbb R^{n\times n}$, $\mathcal N $ is the set of singular $n\times n$ matrices, and
$$
\Pi_T A = U V^t, 
$$
where $A = U\Sigma V^t$ is the singular value decomposition of $A$. 
We have 
$\mathcal B = \{ A \in \mathbb R^{n\times n} \colon \| A \|_s \leq 1 \}$ where $\|\cdot \|_s$ is the spectral norm. See \cite{OnValuedFields} for further details. 

\bigskip

Since $\Omega \subset \mathbb R^d$ is compact, $L^\infty(\Omega; T) \subset L^2(\Omega; T)$. If $T \subset \mathbb R^k$ is compact, then the following Lemma shows that the converse holds.
\begin{lem}
If $T \subset \mathbb R^k$ is compact, then $L^p(\Omega; T) \subset L^\infty(\Omega; T)$ for every $p \geq 1$. 
\end{lem}
\begin{proof}
If $T$ is compact, then there exists $M>0$ such that $|z| \leq M$ for every $z \in T$. Let $f \in L^p(\Omega; T)$ for $p \geq 1$.  Then $\textrm{ess.im}(f) \subset T$ implies that for all $z \in \mathbb R^k\setminus T$, there exists $\varepsilon >0$ such that 
$$
\mu \left( \{ x \in \Omega \colon |f(x) - z| \leq \varepsilon \} \right) = 0.  
$$
In particular, we have that $ \mu \left( \{ x \in \Omega \colon |f(x)| > M \} \right) = 0$ which shows that
$$
\inf \{ a \in \mathbb R \colon  \mu \left( \{ x \in \Omega \colon |f(x)| > a \} \right) = 0 \} \leq M, 
$$
which implies that $f \in L^\infty(\Omega; T)$.
\end{proof}

\bigskip

Since $\langle u, u \rangle = 1$, \eqref{e:EnergyLambda} can be rewritten 
\begin{align*}
E_{\lambda,\tau}(u) 
& =  \frac{1}{2} -  \frac{1}{2} \langle u, e^{\Delta \tau}u \rangle + \frac{\lambda}{2 }\langle u- f, e^{\Delta \tau} (u-f) \rangle \\
& = - \frac{1}{2} (1-\lambda) \langle u, e^{\Delta \tau}u \rangle 
- \lambda \langle u , e^{\Delta \tau}f \rangle 
+  \left( \frac{1}{2} + \frac{\lambda}{2 }\langle f, e^{\Delta \tau} f \rangle \right).
\end{align*}
Ignoring the constant term and multiplying by $-1$, the relaxed problem \eqref{eq:relaxed} becomes
\begin{subequations}
\begin{equation}\label{eq:maxsphere}
\max_{u \in L^\infty(\Omega; T)}  \  \tilde{E}_{\lambda,\tau}(u), 
\end{equation}
where
\begin{equation} \label{e:S2Energy}
\tilde{E}_{\lambda,\tau}(u) := \frac{1}{2} (1-\lambda) \langle u, e^{\Delta \tau} u \rangle + \lambda \langle u,e^{\Delta \tau} f \rangle.
\end{equation}
\end{subequations}
The existence to a solution in \eqref{eq:maxsphere} follows from Theorem~\ref{prop:Existence}.

The following Lemma follows from calculations similar to as in the proof of Lemma \ref{lem:function}.
\begin{lem} \label{lem:function2}
Assume $\lambda \in (0,1)$, $\tau > 0$, and $f \in L^\infty(\Omega;\mathbb R^k)$. Then the following properties hold for the functional $\tilde E_{\lambda,\tau}$ defined in \eqref{e:S2Energy}. 
\begin{enumerate}
\item[(i)] The first variation of $\tilde E_{\lambda,\tau} \colon L^2(\Omega; \mathbb R^k) \to \mathbb R$ with respect to $u$ is 
\begin{equation} \label{e:L}
L_{\lambda,\tau}^{u}(v)  :=  \left\langle v,  \frac{\delta \tilde E_{\lambda,\tau}(u)}{\delta u}  \right\rangle =  \langle v, (1-\lambda) e^{\Delta \tau} u + \lambda e^{\Delta \tau} f  \rangle . 
\end{equation}
\item[(ii)] $\tilde E_{\lambda,\tau}(u)$ is  convex on $L^2(\Omega;\mathbb R^k)$.
\end{enumerate}
\end{lem}

Since the maximum in \eqref{eq:maxsphere} is attained by an extremal point of $ L^\infty(\Omega; \mathcal B)$, \ie, $ L^\infty(\Omega; T)$, we have the following Lemma.
\begin{lem} 
The optimization problem in \eqref{eq:maxsphere} is equivalent to 
\begin{equation} \label{e:sphereEquiv}
\max_{u \in L^\infty(\Omega; \mathcal B) } \   \tilde{E}_{\lambda,\tau}(u).
\end{equation}
\end{lem}

The sequential linear programming approach to solving \eqref{e:sphereEquiv}  is to consider a sequence of functions
$\{u_n\}_{n=0}^{\infty}$ which satisfies
\[u_{n+1} = \arg\max_{u\in L^\infty(\Omega; \mathcal B)} L_{\lambda,\tau}^{u_n}(u), \qquad  u_0 \in L^\infty(\Omega; T) \, \textrm{is given}\]
where the linear functional $L_{\lambda,\tau}^{u_n}$  is given in \eqref{e:L}. 
\begin{lem} \label{lem:maxL} 
If $e^{\Delta \tau} ( (1-\lambda) u_n +\lambda  f )(x) \notin \mathcal N$ a.e. $x \in \Omega$, then the 
 maximizer of the linear functional $L_{\alpha,\tau}^{u_n}(u)$ over $L^\infty(\Omega; T)$ is 
\[u_{n+1} =  \Pi_T \left( e^{\Delta \tau} ( (1-\lambda) u_n +\lambda  f ) \right). \] 
\end{lem}
\begin{proof}
Writing $w = e^{\Delta \tau} ( (1-\lambda) u_n +\lambda  f )$, we have $L_{\lambda,\tau}^{u_n}(u) = \langle u, w\rangle$, it follows that 
$$
\arg \max_{u\in L^\infty(\Omega; T) }  \  \langle u, w\rangle 
 \ \ = \ \ 
 \arg \min_{u\in L^2(\Omega; T)} \ \| u - w \|^2. 
$$
Since $w \in C^\infty$ and using the definition of $\Pi_T$, the result follows. 
\end{proof}
The iterates in Lemma \ref{lem:maxL} are equivalent to those in Algorithm \ref{a:alg1}. The following theorem gives the stability of Algorithm \ref{a:alg1}. The proof can be adapted directly from \cite[Prop. 4.3]{OnValuedFields}. 
\begin{thm} \label{p:Lyapunov}
Assume $T$ is a closed subset of $\mathbb S^{k-1}$ such that the closest point mapping, $\Pi_T$ is defined on $\mathbb R^k \setminus \mathcal N$. Let $\lambda \in (0,1)$, $\tau > 0$, and $f \in L^\infty(\Omega;\mathbb R^k)$. 
If $e^{\Delta \tau} ( (1-\lambda) u_n +\lambda  f )(x) \notin \mathcal N$ a.e. $x \in \Omega$ and $n\in \mathbb N$, 
the functional $E_{\lambda,\tau}$, defined in \eqref{eq:relaxed}, is non-increasing on the iterates $\{ u_n \}_{n=1}^\infty$, {\it i.e.}, $E_{\lambda,\tau}(u_{n+1}) \leq E_{\lambda,\tau}(u_n)$. 
\end{thm}

\begin{rem}
Regarding the condition assumed in Theorem~\ref{p:Lyapunov}, in practice, we do not observe that $e^{\Delta \tau} ( (1-\lambda) u_n +\lambda  f )(x) \in \mathcal N$ for any $x \in \Omega$ or $n\in \mathbb N$. If this condition did occur, a random closest point could be assigned by $\Pi_T$. 
\end{rem}

\begin{rem}
For the special case that $T = \mathbb S^0 = \{ \pm 1 \}$, a proof of convergence for a discrete version of Algorithm \ref{a:alg1} can be adapted from \cite[Prop. 4.4]{OnValuedFields}. 
\end{rem}

\subsection{Energy splitting interpretation of Algorithm~\ref{a:alg1}} \label{s:EnergySplit}
In this section, we interpret Algorithm~\ref{a:alg1} as a splitting method for  \eqref{eq:relax}.  

Define the {\sl proximal operator for the functional $E\colon L^2(\Omega; \mathbb R^k)\to \mathbb R$}, 
${\bf Prox}_E\colon L^2(\Omega;\mathbb{R}^k) \to L^2(\Omega; \mathbb{R}^k)$, by
\[{\bf Prox}_E(v) := \arg\min_{u\in L^2(\Omega; \mathbb R^k)} \ E(u)+\frac{1}{2} \|u-v\|^2,
\]
and, for $\beta>0$, similarly define the {\sl proximal operator of the scaled functional $\beta E \colon L^2(\Omega; \mathbb R^k)\to \mathbb R$}, ${\bf Prox}_{\beta E}\colon L^2(\Omega;\mathbb{R}^k) \to L^2(\Omega; \mathbb{R}^k)$, by
\[{\bf Prox}_{\beta E}(v) := \arg\min_{u\in L^2(\Omega; \mathbb R^k)} \ E(u)+\frac{1}{2\beta} \|u-v\|^2.
\]
To simplify notation, we rewrite \eqref{eq:relax} by 
\begin{equation} \label{e:Esplit}
\inf_{u \in H^1(\Omega;\mathbb R^k )}  \ \mathcal{E}_1(u)+\mathcal{E}_2(u)+\mathcal{E}_3(u)
\end{equation}
where $\mathcal{E}_1(u) = \frac{1}{2} \int_\Omega  |\nabla u(x) |^2 \ dx$, $ \mathcal{E}_2(u) =  \frac{1}{2 \alpha} \int_{\Omega} |u(x) - f(x)|^2 \ dx$, and $ \mathcal{E}_3(u) =  \frac{1}{\varepsilon^2}  \int_\Omega L \left( u(x) \right) \ dx$.

Now, we introduce an iterative method by formally splitting the minimization in \eqref{e:Esplit} as follows. Let $u_n$ be the approximation to the solution at the $n$-th iteration. We first consider $ \mathcal{E}_2$ and solve
\begin{equation}\label{e:energysplitting1}
u_n^\star =  {\bf Prox}_{\tau \mathcal{E}_2}(u_n).
\end{equation}
We then evolve $u_n^\star$ by the gradient flow of $\mathcal{E}_1$ until time $t = \tau$,
\begin{equation} \label{e:energysplitting2}
\begin{cases}
v_t = \Delta v ,\\
v|_{t=0}= u_n^\star
\end{cases}
\end{equation}
to obtain $u_n^{\star\star} = v(\tau)$.
Finally, we consider $\mathcal E_3$ and set
\begin{equation}\label{e:energysplitting3}
u_{n+1} =  {\bf Prox}_{\tau \mathcal{E}_3}(u_n^{\star\star}).
\end{equation}

We can directly solve \eqref{e:energysplitting1}, by setting the first variation to zero,
$$
\frac{\delta (\mathcal{E}_2(u) + \frac{1}{2\tau}\|u-u_n\|^2)}{\delta u} = \frac{1}{\alpha} (u-f)+\frac{1}{\tau} (u-u_n) = 0,
$$
and solving for $u$ to obtain 
$$
u_n^\star = (1- \tilde \lambda)u_n + \tilde \lambda f, 
\qquad \textrm{where} \ \tilde \lambda = \frac{\tau}{\alpha+\tau} = \frac{\lambda}{1 + \lambda}.
$$
This together with \eqref{e:energysplitting2} gives the diffusion step in Algorithm~\ref{a:alg1} for a modified convex combination parameter. 

As $\varepsilon \searrow 0$, the solution to \eqref{e:energysplitting3} is obtained by taking a function $u \in L^2(\Omega; T)$ such that  \[u = \arg\min_{u \in L^2(\Omega; T)} \ \|u-u_n^{\star\star} \|^2,\] which is precisely the point-wise projection, 
$$
u_{n+1} = \Pi_T(u_n^{\star\star}).
$$ 
This is the projection step in Algorithm~\ref{a:alg1}.

\section{Computational examples} \label{s:CompEx}
In this Section, we demonstrate the diffusion generated method in Algorithm~\ref{a:alg1}, developed  in Section~\ref{s:alg}, on several synthetic and test numerical experiments. 
Several of the numerical experiments considered are from  \cite{Bacak_2016}, which can provide a comparison. 
For these examples, we both cite the section number and include details of the examples for completeness. 

The data in Sections~\ref{s:lemniscate} and \ref{s:S2} is periodic, so we solve the heat diffusion equation in Algorithm \ref{a:alg1} 
with a periodic boundary condition instead of the free-space heat diffusion equation. 
For all other examples, we solve the free-space heat diffusion equation.
All methods were implemented in MATLAB and results reported below were obtained on a laptop with a 2.7GHz Intel Core i5 processor and 8GB of RAM. 

\subsection{Example: a synthetic $\mathbb{S}^2$-valued one-dimensional image} \label{s:lemniscate}
Following \cite[\S5.1]{Bacak_2016}, we consider the lemniscate of Bernoulli, given by 
$$
\gamma(t)  = \frac{\pi/2}{\sin^2(t)+1} \left( \cos(t), \cos(1) \sin(t) \right), \qquad t \in [0,2\pi].
$$
For $p \in \mathcal{M}$ and  $\xi \in T_p \mathcal{M}$, let $\Gamma_{p,\xi}$ be the unique geodesic starting from $\Gamma_{p,\xi}(0) = p$ with $\dot{\Gamma}_{p,\xi}(0) = \xi$. Then we define $\exp_p\colon T_p \mathcal{M} \to \mathcal{M}$ by $\exp_p(\xi) = \Gamma_{p,\xi}(1)$. Then, one spherical lemniscate curve can be obtained by
\[\gamma_S(t) = \exp_p(\gamma(t))\]
with $p = (0,0,1)$. We discretize the curve in the parameter space of $t$ at $t_i := \frac{2\pi i}{511}, i = 0, \ldots , 511$ to get an $\mathbb S^2$-valued signal, which we denote by $\{f_{o,i}\}_{i=0}^{511}$. We add noise to the data by taking  $f_i = \exp_{f_{o,i}}(\eta_i)$ with $\eta_i = (\eta_{i1},\eta_{i2})$, 
where $\eta_{i1}$ and $\eta_{i2}$ are independent Gaussian noises with standard deviation of $\frac{\pi}{30}$. In  Figure~\ref{f:lemniscate}(a), the blue markers indicate the original data and the red markers indicate the noisy data.

In this example, we take the target set, $T$, to be $\mathbb{S}^2 = \{ x\in \mathbb{R}^3\colon |x| = 1\}$. 
Then, we have 
$\mathcal N = \{x \in \mathbb R^2 \colon x = 0\}$ and 
$
\Pi_T (x) = 
\frac{x}{|x|} \  \textrm{if }  \  x\neq 0 .
$
Applying Algorithm \ref{a:alg1}, we get the denoised results shown in Figures~\ref{f:lemniscate}(b)--\ref{f:lemniscate}(d) with a fixed $\tau = 10^{-3}$ and $\lambda = 0.05$, $0.1$, and $0.15$, respectively. Since the original image is periodic, we solve the diffusion equation in Algorithm \ref{a:alg1} with the periodic boundary condition.

We observe that the denoised results very closely match the original data and that the results are relatively insensitive to the parameter $\lambda$. 
All of these simulations were completed within $0.01$ seconds. 

\begin{figure}[t]
\begin{center}
\subfigure[Original (blue) and noisy data (red)]{\includegraphics[width=.45\textwidth]{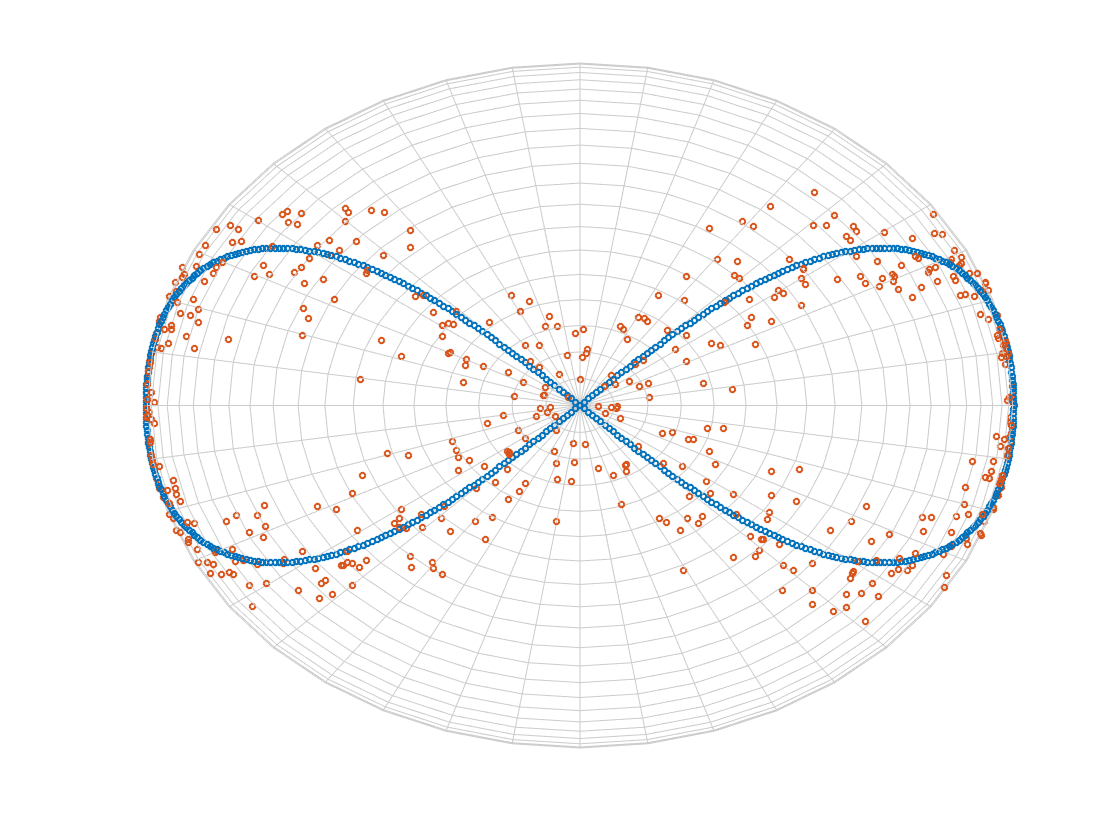}} 
\subfigure[Denoised data with $\lambda = 0.05$ and $\tau = 10^{-3}$]{\includegraphics[width=.45\textwidth]{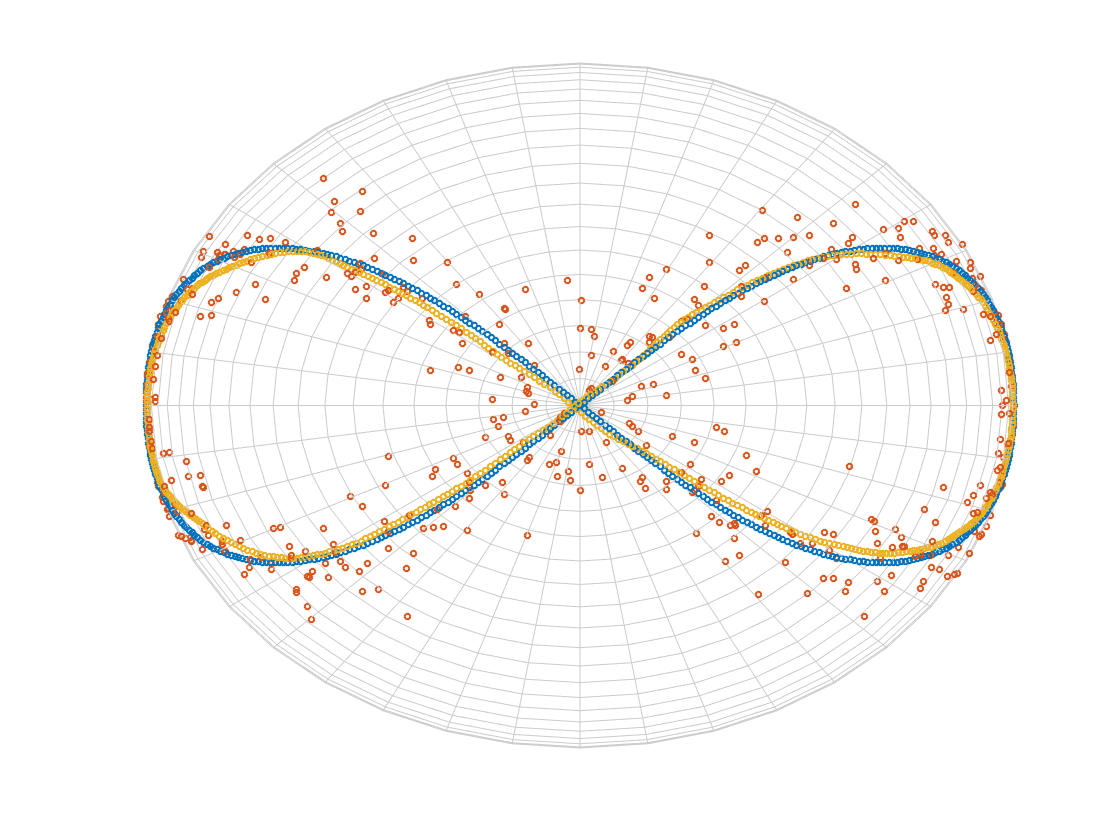}}
\subfigure[Denoised data with $\lambda = 0.1$ and $\tau =10^{-3}$]{\includegraphics[width=.45\textwidth]{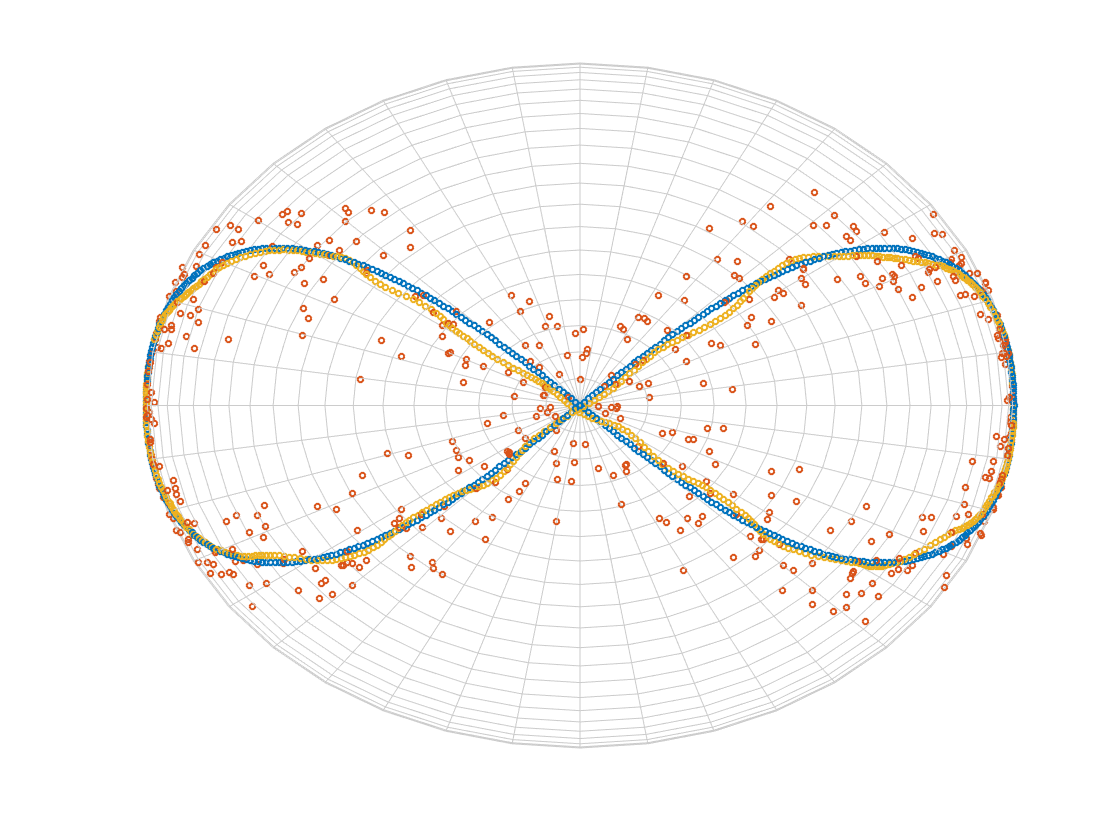}}
\subfigure[Denoised data with $\lambda = 0.15$ and $\tau = 10^{-3}$]{\includegraphics[width=.45\textwidth]{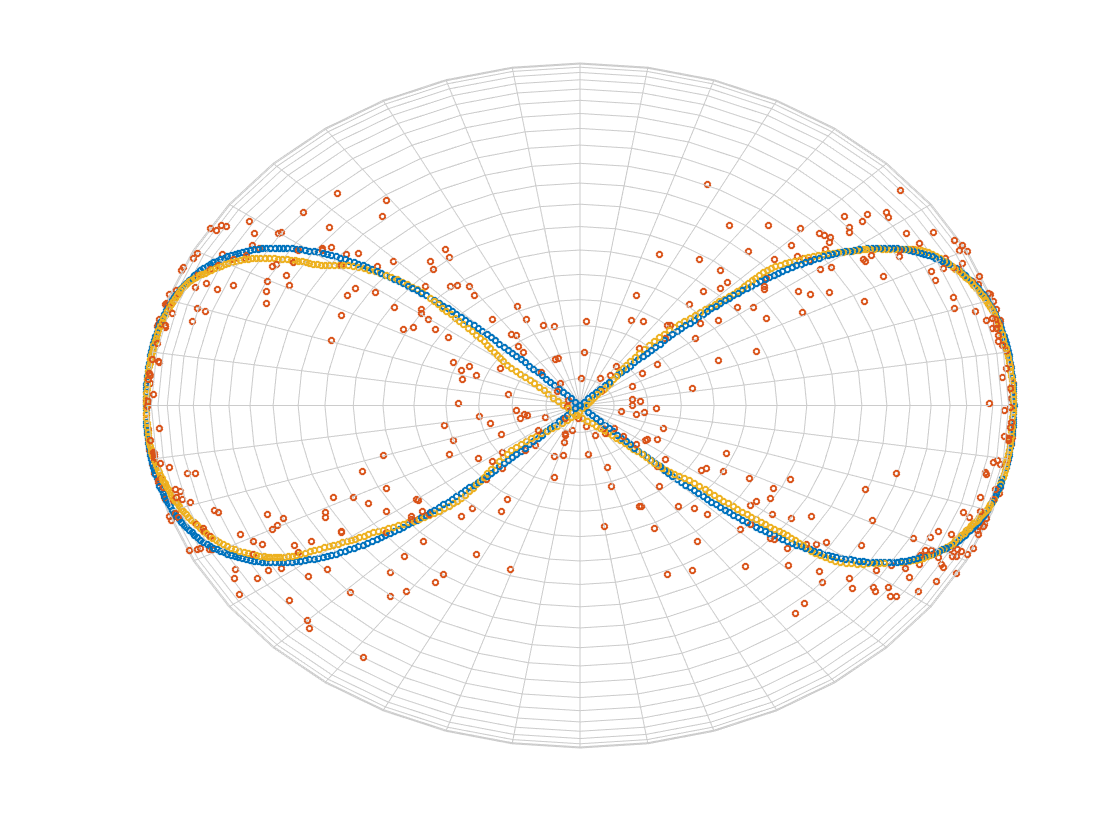}}
\caption{Results of denoising  an obstructed lemniscate of Bernoulli on the sphere, $\mathcal{S}^2$, with $\lambda = 0.05, 0.1$ and $ 0.15$, respectively. In this simulation, $\tau$ is fixed as $10^{-3}$. See Section~\ref{s:lemniscate}. }
\label{f:lemniscate}
\end{center}
\end{figure}

\subsection{Example: a synthetic $\mathbb{S}^2$-valued image} \label{s:S2}
Again, following \cite[\S5.1]{Bacak_2016}, we define an $\mathbb{S}^2$-valued vector-field
by
\[G(t, s) = R_{x(t)+y(s)}S_{x(t)-y(s)} e_3, 
\qquad \quad  
t\in [0, 8 \pi ],  \ \ s \in [0, 8 \pi ]
\]
where 
\[x(t) = t+\frac{\pi}{2} \left\lfloor \frac{t}{2\pi} \right\rfloor, \qquad  y(t) = t+\frac{\pi}{2} \left\lfloor \frac{t}{2\pi} \right\rfloor,\] 

\[ R_\theta := \left[\begin{matrix}
\cos \theta &−\sin \theta & 0 \\
\sin \theta & \cos \theta &0 \\
0 &0 & 1 
\end{matrix} \right], 
\qquad \textrm{and} \qquad 
S_\theta := \left[\begin{matrix}
\cos \theta & 0 &−\sin \theta  \\
0 & 1 & 0\\
\sin \theta & 0  & \cos \theta 
\end{matrix} \right]. 
\]

We sample both dimensions, $t$ and $s$, with $n = 64$ points to obtain a discrete vector field
$f_o \subset \mathbb{S}^2$ which is given in Figure~\ref{f:sphere_valued_image}(a). Similar to the noise in \ref{s:lemniscate}, we add the Gaussian noise in the tangential plane of each point with standard deviation $\frac{4\pi}{45}$; the noisy data is displayed in Figure~\ref{f:sphere_valued_image}(b).

\begin{figure}[t]
\begin{center}
\subfigure[Original data.]{\includegraphics[width=.48\textwidth,clip,trim= 4cm 4cm 4cm 4cm]{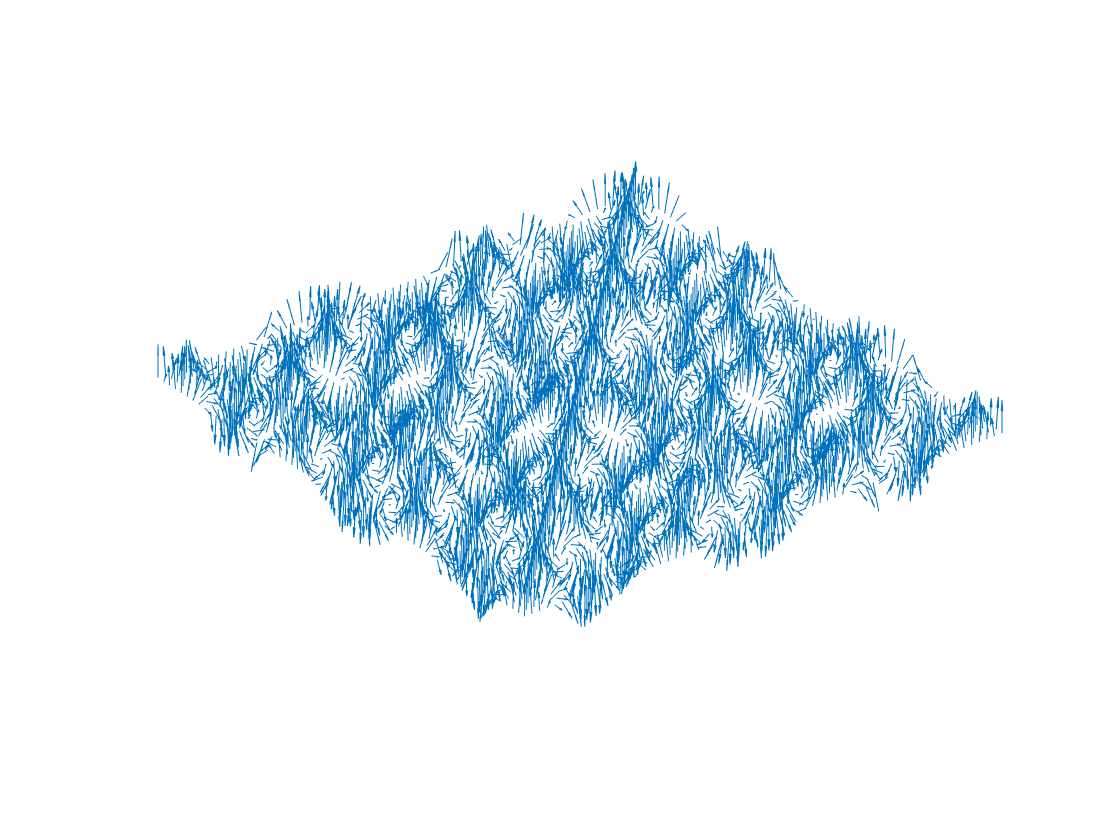}}
\subfigure[Noisy data.]{\includegraphics[width=.48\textwidth,clip,trim= 4cm 4cm 4cm 4cm]{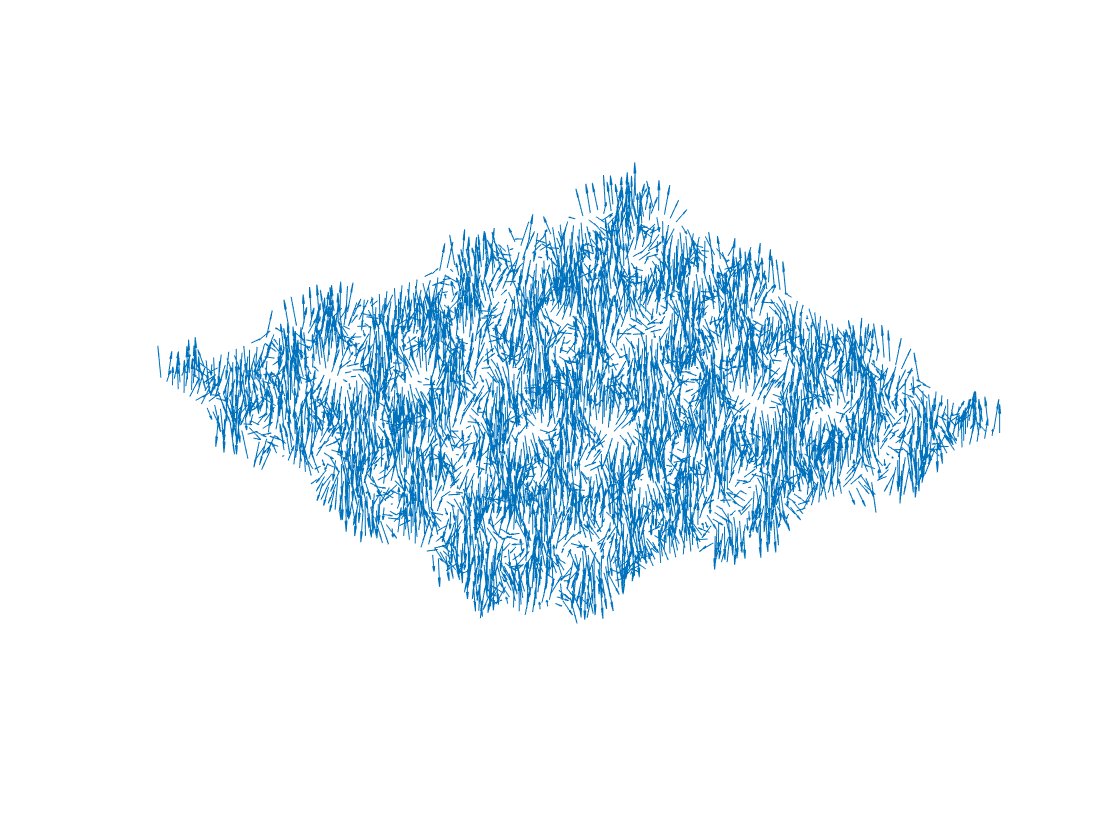}}
\subfigure[Denoised data with $\lambda = 0.05$ and $\tau = 10^{-3}$.]{\includegraphics[width=.48\textwidth,clip,trim= 4cm 4cm 4cm 4cm]{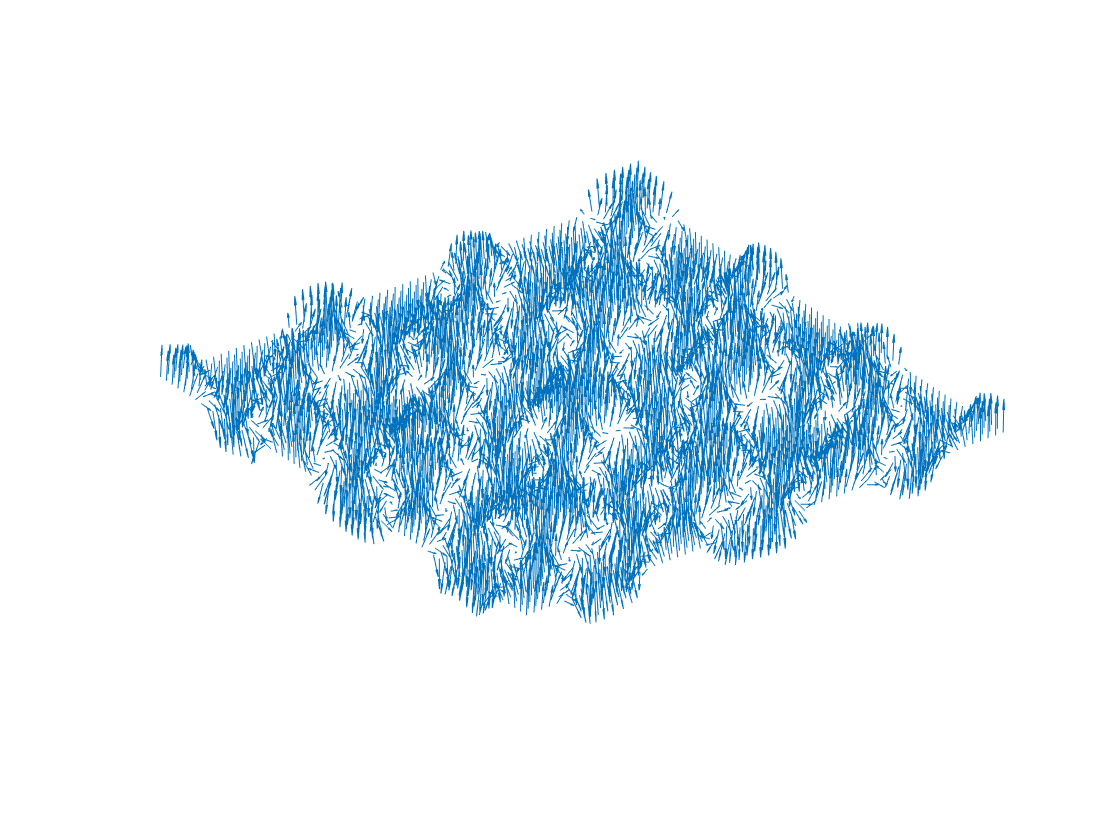}}
\subfigure[Denoised data with $\lambda = 0.1$ and $\tau = 10^{-3}$.]{\includegraphics[width=.48\textwidth,clip,trim= 4cm 4cm 4cm 4cm]{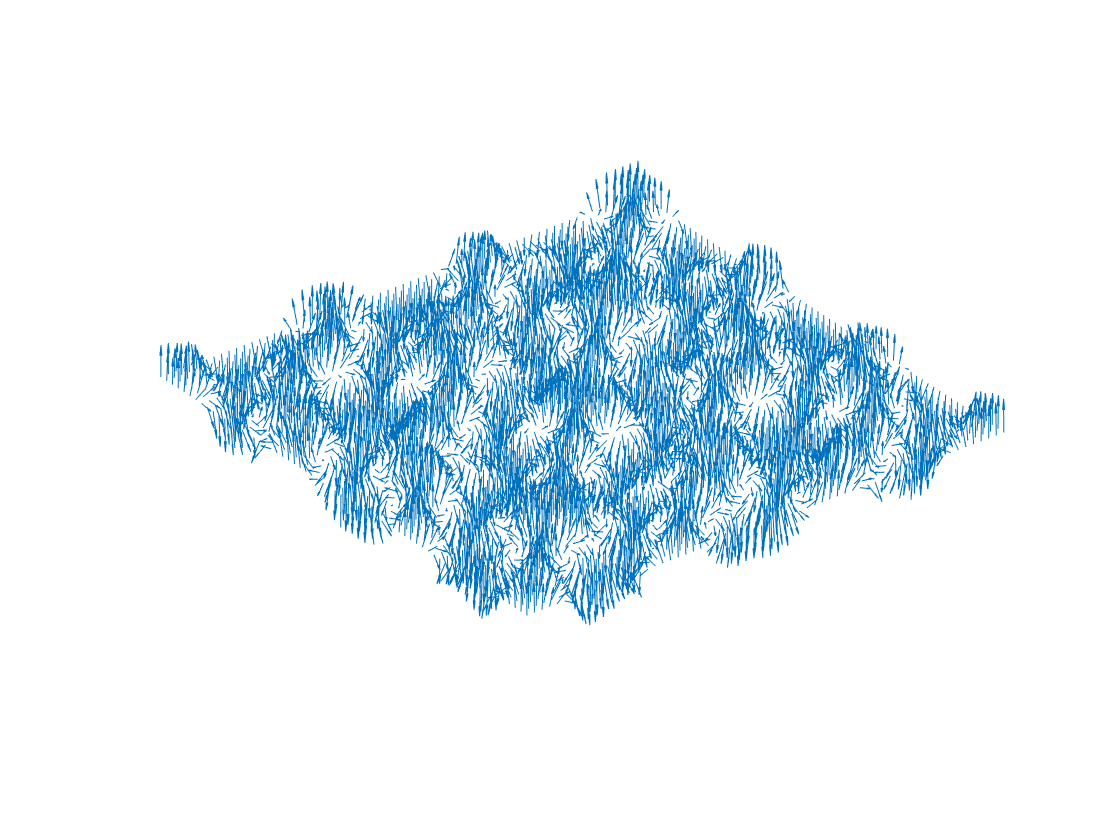}}
\subfigure[Denoised data with $\lambda = 0.15$ and $\tau = 10^{-3}$.]{\includegraphics[width=.48\textwidth,clip,trim= 4cm 4cm 4cm 4cm]{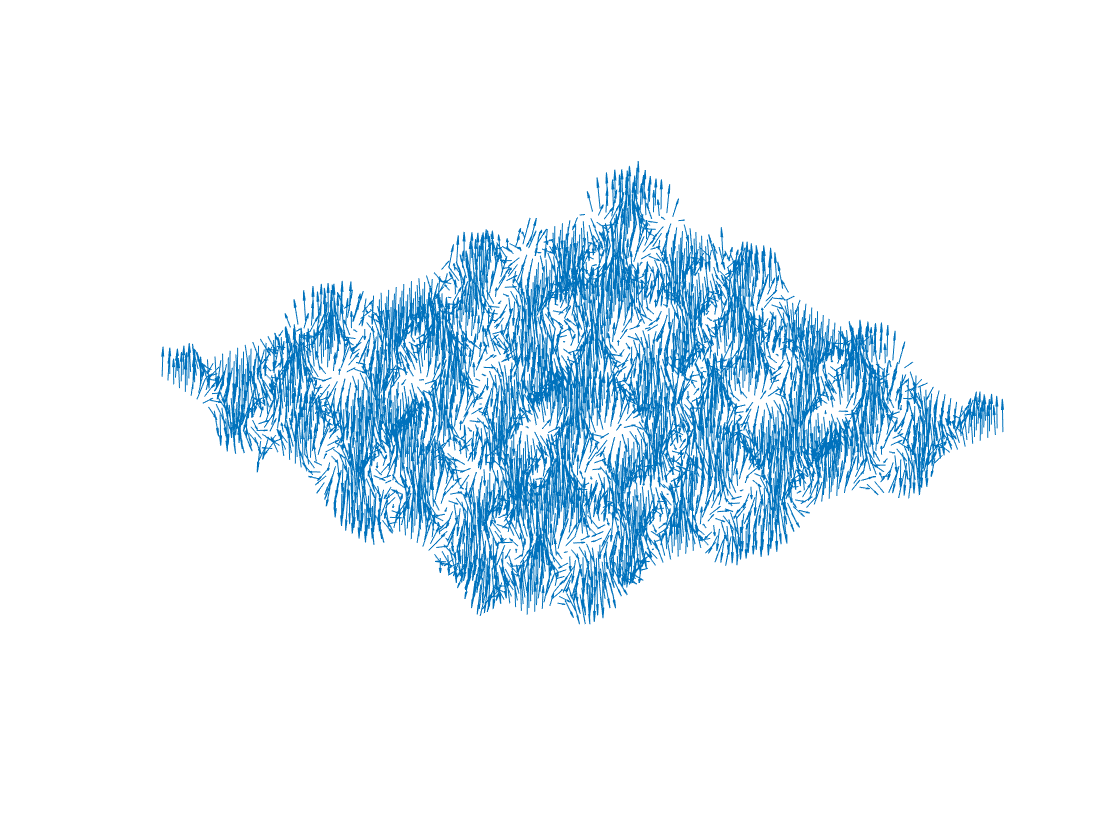}}
\subfigure[Denoised data with $\lambda = 0.2$ and $\tau = 10^{-3}$.]{\includegraphics[width=.48\textwidth,clip,trim= 4cm 4cm 4cm 4cm]{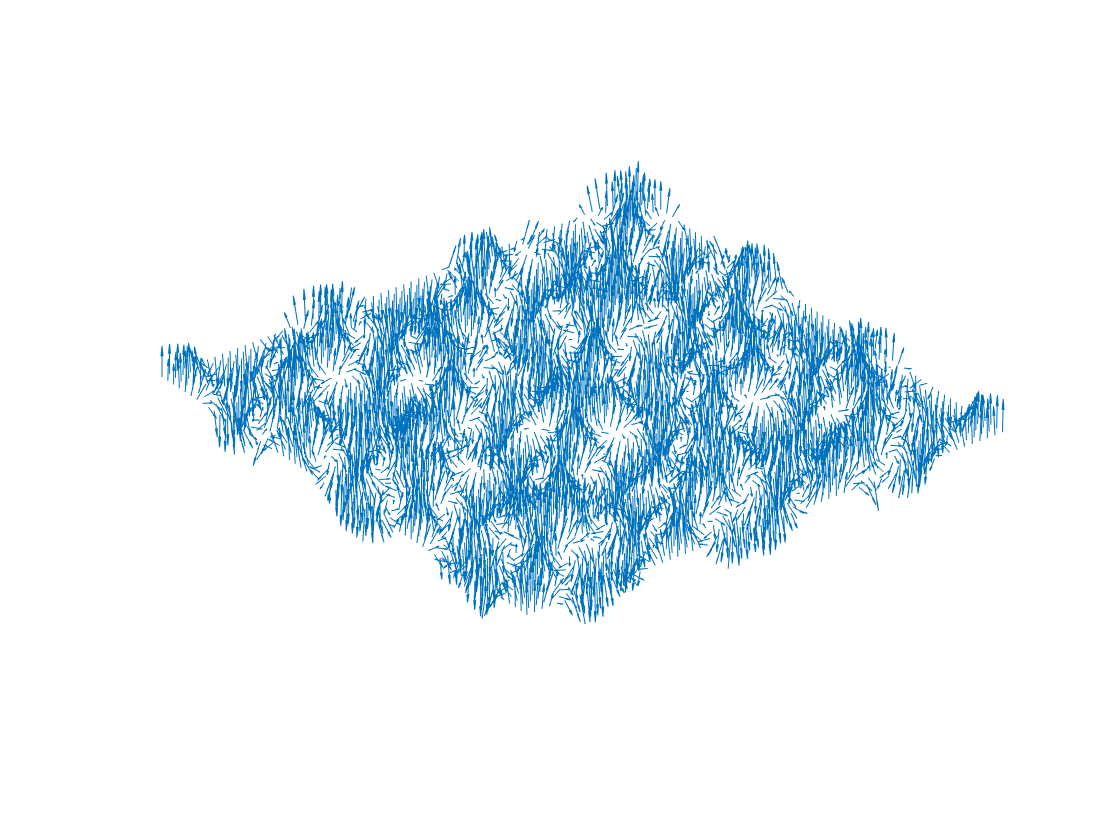}}
\caption{ Results of the reconstruction of a noisy $\mathbb{S}^2$-valued image with $\lambda = 0.05$, $0.1$, $0.15$, and $0.2$, respectively. In all the simulations, $\tau$ is fixed as $10^{-3}$. See Section~\ref{s:S2}.}
\label{f:sphere_valued_image}
\end{center}
\end{figure}

In this example, we take the target set, $T$, to be $\mathbb{S}^2 = \{ x\in \mathbb{R}^3\colon |x| = 1\}$. 
Then, we have 
$\mathcal N = \{x \in \mathbb R^3 \colon x = 0\}$ 
and 
$\Pi_T (x) = \frac{x}{|x|} \  \textrm{if }  \  x\neq 0$. Since the original image is periodic, we solve the diffusion equation in Algorithm \ref{a:alg1} with the periodic boundary condition.

The  results of Algorithm \ref{a:alg1} with $\tau = 10^{-3}$ and $\lambda = 0.05$, $0.1$, $0.15$, and $0.2$ are displayed in Figure~\ref{f:sphere_valued_image}(c)--(f). The numerical results show that Algorithm \ref{a:alg1} is robust with respect to the value of the parameter $\lambda$ and it is applicable to recover the original data image in Figure~\ref{f:sphere_valued_image}(a). It is also observed that the denoised image is slightly smoothed when decreasing the value of $\lambda$. Again, the algorithm performs very efficiently; all simulations can be done within 0.2 seconds.

\subsection{Example: the `peppers' image} \label{s:Peppers}
Following \cite[\S5.1]{Bacak_2016}, we consider the denoising of the `peppers' image, shown in Figure~\ref{f:Pepper_data}(a). It is distorted with Gaussian noise in each of the red, green, and blue (RGB) channels with standard deviation as $0.1$, as show in Figure~\ref{f:Pepper_data}(b).

\begin{figure}
\subfigure[Original image]{\includegraphics[width=.45\textwidth]{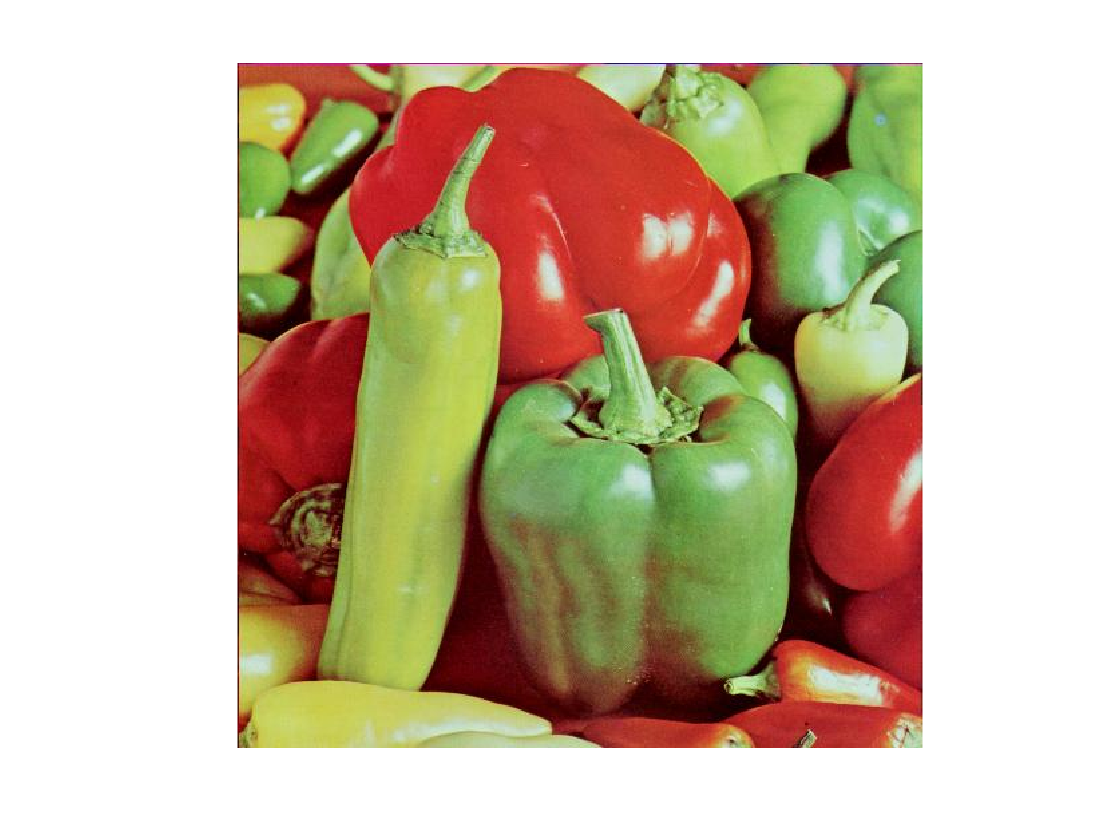}}
\subfigure[Noisy image]{\includegraphics[width=.45\textwidth]{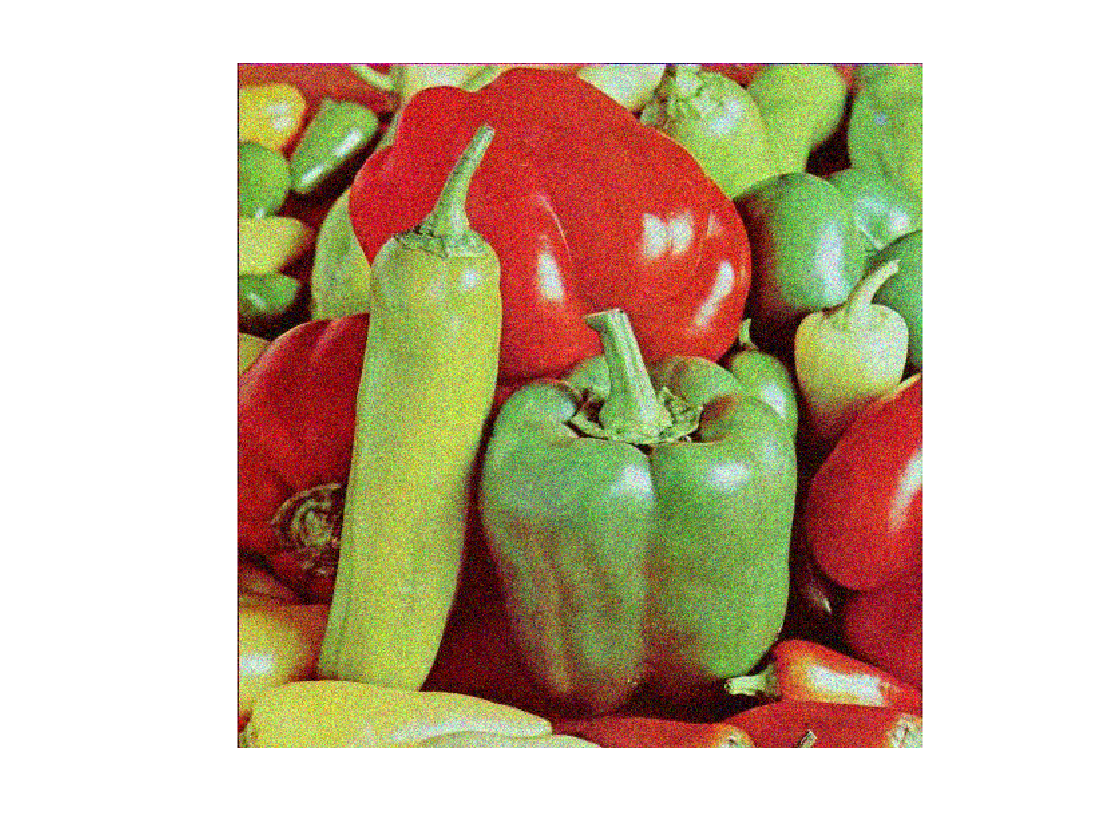}}
\caption{The original `peppers' image and one corrupted with noise. See Section~\ref{s:Peppers}.}\label{f:Pepper_data}
\end{figure}

We first consider the image represented in RGB channels. Here, the target set is the unit cube, $T =  [0,1]^3 \subset \mathbb R^3$. Since this is a convex set, we let $\Pi_T$ be the convex projection and apply Algorithm~\ref{a:alg1} to  denoise Figure~\ref{f:Pepper_data}(b). Figure~\ref{f:Pepper_rgb} displays the denoised result for different values of the parameter $\lambda$. 

We next consider the image represented in hue, saturation, and value (HSV) channels. 
Here, the color space, $(H,S,V) \in S^1 \times [0,1] \times [0,1]$,  consists of a  $\mathbb{S}^1$-valued hue component, $H$, and two $[0,1]$-valued components: saturation, $S$, and value, $V$. For this product target space, $T$, we define the mapping $\Pi_T$ to map onto each component individually.  The results of applying Algorithm~\ref{a:alg1} to  denoising Figure~\ref{f:Pepper_data}(b) in the HSV channels are displayed in Figure~\ref{f:Pepper_hsv} for different values of the parameter $\lambda$.  

We use a relatively large value of $\lambda$ compared to other numerical experiments since the original image is non-smooth. 
We observe that the denoised images are slightly blurred, but the results are robust to changes in the parameter $\lambda$. 
We use the peak signal-to-noise ratio (PSNR) to evaluate the quality of the denoised image. The PSNR obtained in  \cite{Bacak_2016} for both RGB and HSV channels ranges between 28.16 to is 31.24, which are slightly better than the values obtained using this much simpler method. 

\begin{figure}[ht]
\begin{center}
\subfigure[$\lambda=0.85$, $ \tau=10^{-4}$, PSNR$=28.3314$.]{\includegraphics[width=.32\textwidth,clip,trim= 6cm 1cm 6cm 1cm]{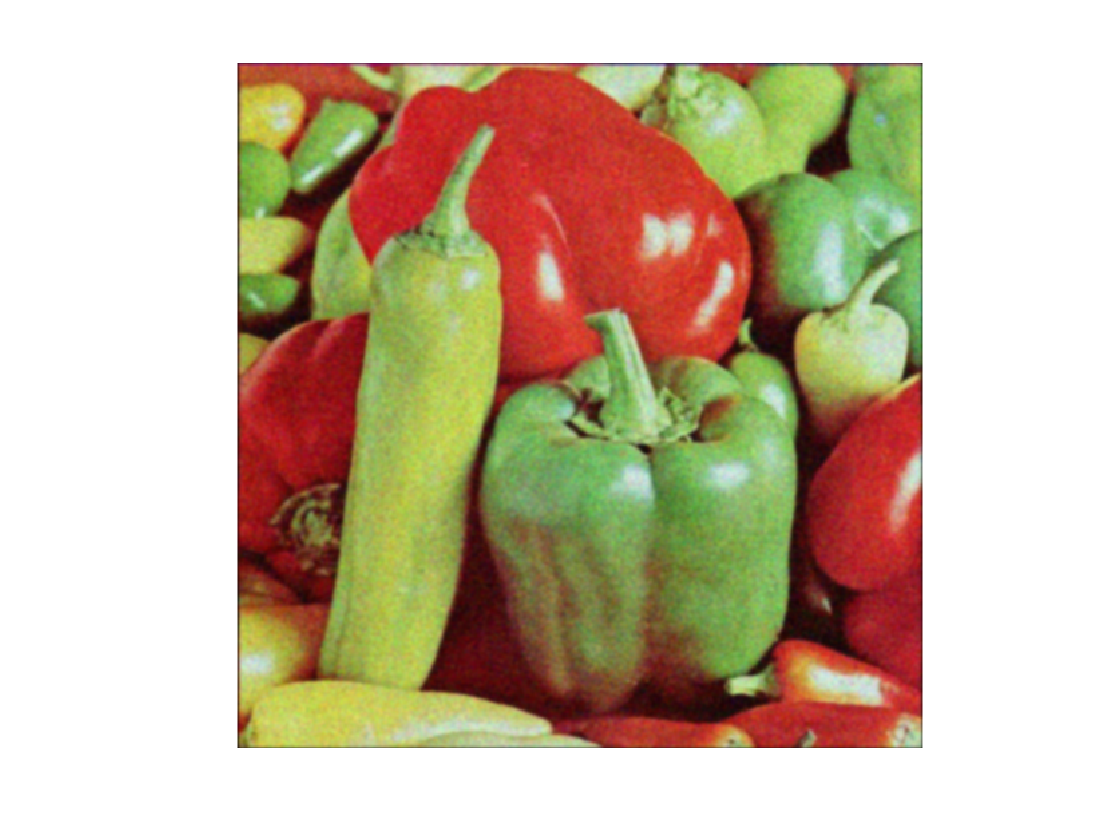}}
\subfigure[$\lambda=0.9$, $\tau = 10^{-4}$, PSNR$=28.3754$.]{\includegraphics[width=.32\textwidth,clip,trim= 6cm 1cm 6cm 1cm]{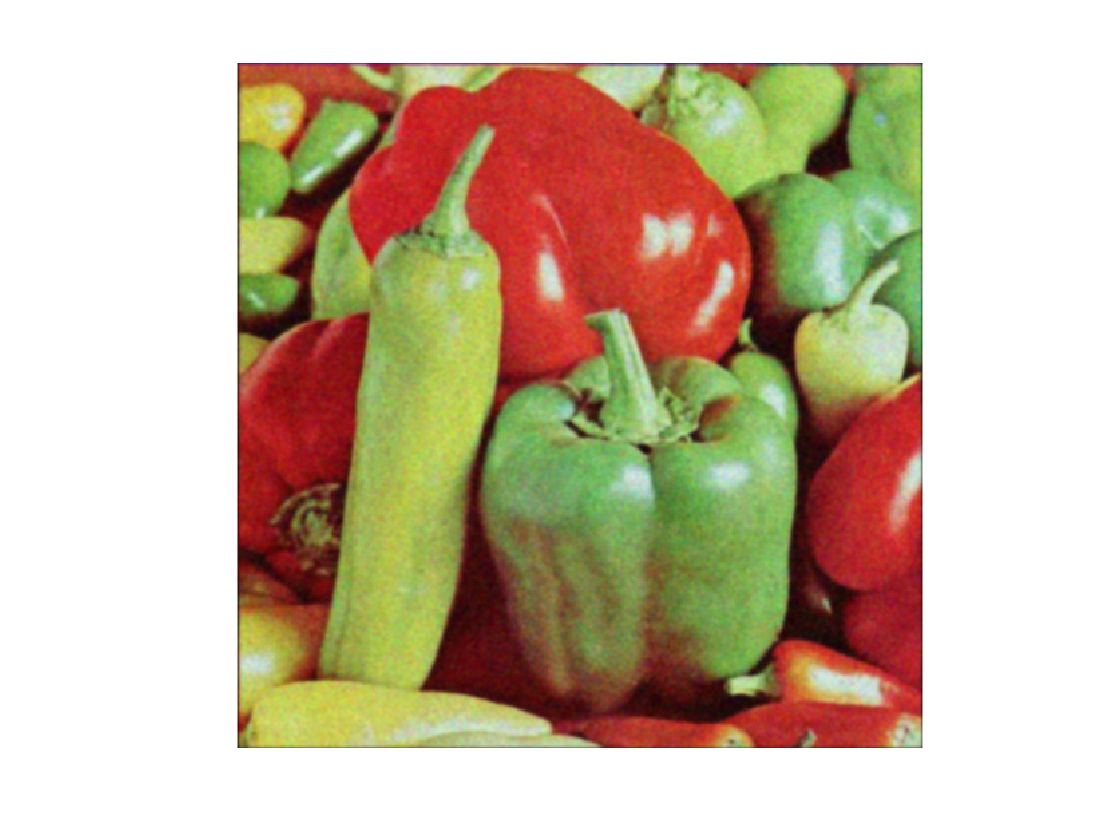}}
\subfigure[$\lambda = 0.95$,$\tau = 10^{-4}$, PSNR$=28.4118$.]{\includegraphics[width=.32\textwidth,clip,trim= 6cm 1cm 6cm 1cm]{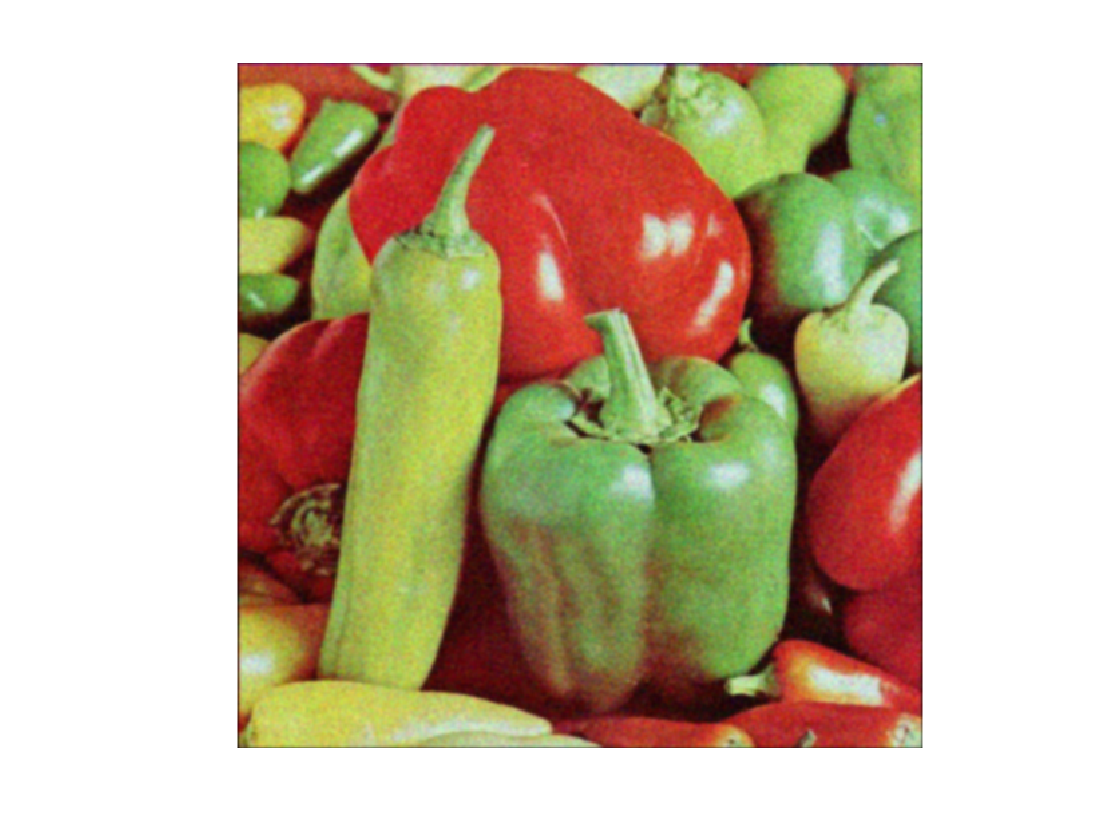}}
\caption{Denoising in the RGB channel on the noisy `peppers' image (see Figure~\ref{f:Pepper_data}(b)) with $\tau = 10^{-4}$ and $\lambda = 0.85$, $0.9$, and $0.95$. The PSNR listed below each image indicates the quality of  the result. See Section~\ref{s:Peppers}.}
\label{f:Pepper_rgb}
\end{center}
\end{figure}

\begin{figure}[ht]
\begin{center}
\subfigure[$\lambda=0.85$, $\tau=10^{-4}$, PSNR$=26.3967$.]{\includegraphics[width=.32\textwidth,clip,trim= 6cm 1cm 6cm 1cm]{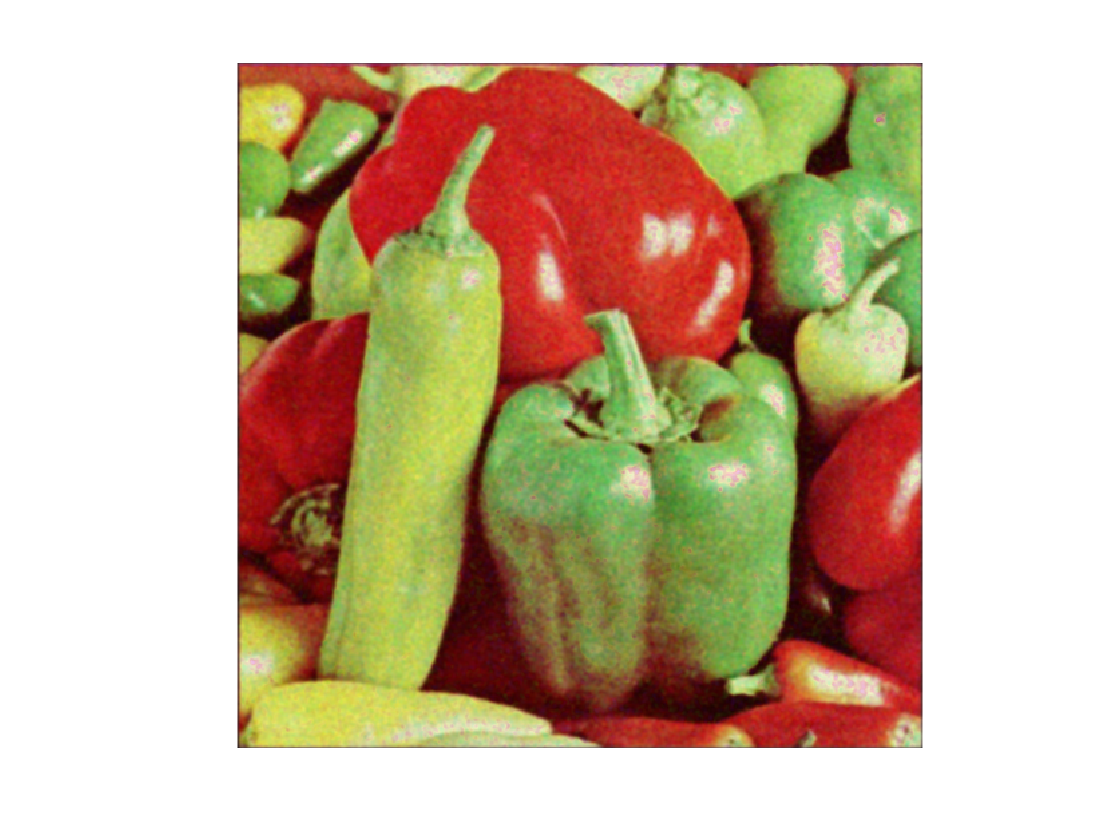}}
\subfigure[$\lambda=0.9$, $\tau = 10^{-4}$, PSNR$=26.4092$.]{\includegraphics[width=.32\textwidth,clip,trim= 6cm 1cm 6cm 1cm]{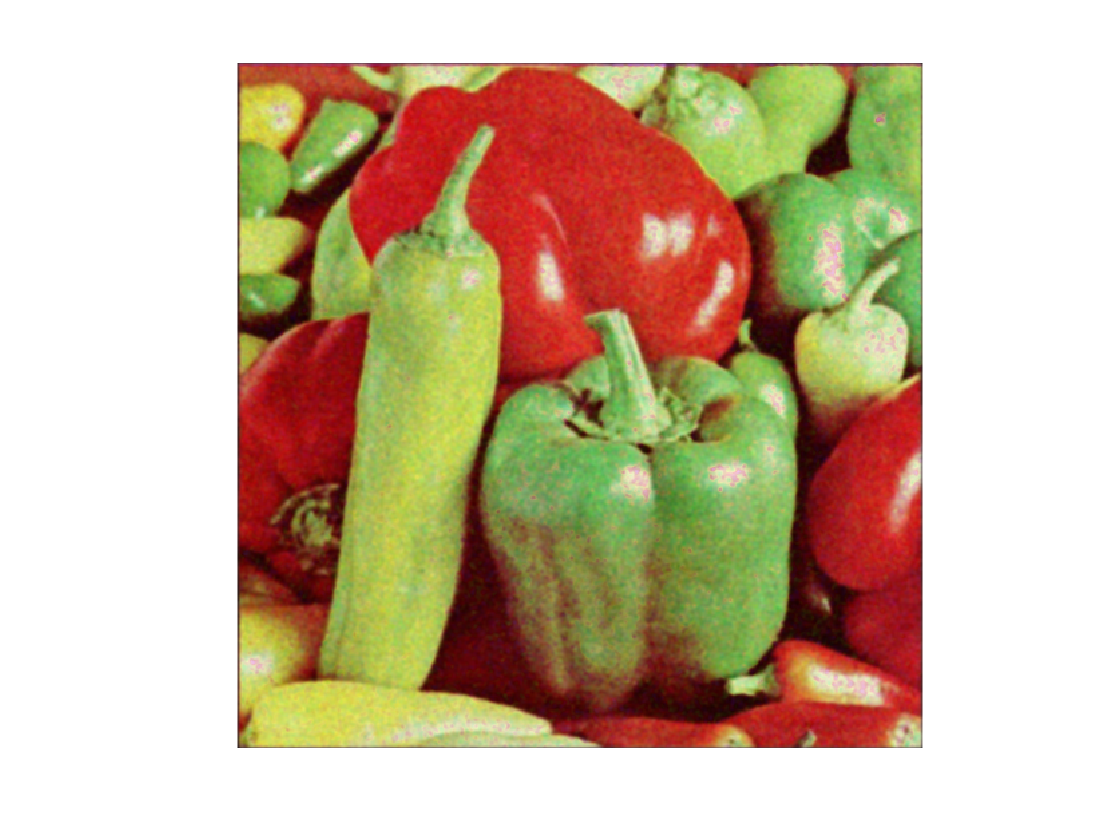}}
\subfigure[$\lambda=0.95$, $\tau = 10^{-4}$, PSNR$=26.4179$.]{\includegraphics[width=.32\textwidth,clip,trim= 6cm 1cm 6cm 1cm]{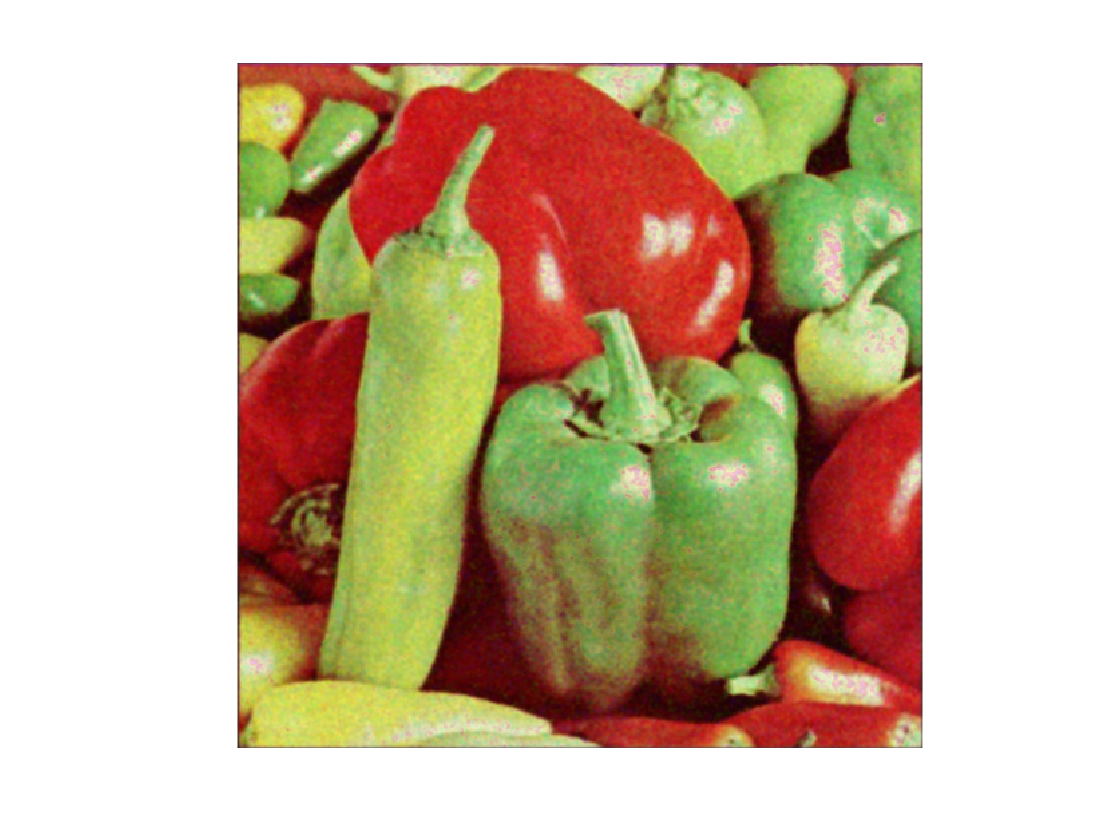}}
\caption{Denoising in the HSV channel on the noisy `peppers' image image (see Figure~\ref{f:Pepper_data}(b)) with $\tau = 10^{-4}$ and $\lambda = 0.85$, $0.9$, and $0.95$. The PSNR are listed below each image to indicate the quality of  the result.See Section~\ref{s:Peppers}.}
\label{f:Pepper_hsv}
\end{center}
\end{figure}

\subsection{Example: an $\mathrm{SPD}(3)$-valued image} \label{s:SPD}
Following \cite[\S5.2]{Bacak_2016}, we consider the reconstruction of a synthetic symmetric positive definite ($\mathrm{SPD}(3)$) matrix-valued image.

An  $\mathrm{SPD}(3)$-valued image is constructed by sampling
$$
G(s,t) := A(s,t) \ \textrm{diag}[v(s,t)] \  A(s,t)^t, 
\qquad \textrm{where} \  s,t \in [0,1].
$$
Here, we define 
\begin{align*}
A(s,t) &= R_{x_2,x_3}(\pi s) R_{x_1,x_2}(|2\pi s-\pi|)R_{x_1,x_2}(|\pi(t-s-|t-s|)-\pi|), \\
v(s,t) &= (1+\delta_{x+y,1}, 1+s+t+\frac{3}{2}\delta_{s,\frac{1}{2}},4-s-t+\frac{3}{2}\delta_{t,\frac{1}{2}})^T
\end{align*}
where $R_{x_i,x_j}(t)$ is the rotation matrix in the $x_i, x_j$-plane with angle $t$ and $\delta_{a,b} =\begin{cases}
1 & \textrm{if} \quad a>b \\
0 & \textrm{else}
\end{cases}$ .

We discretize the parameter space $(s,t)$ with $25\times 25$ grid points to obtain a $25\times 25$ matrix-valued image $f = \{f_{i,j}\}_{i,j=1}^{25} \subset \textrm{SPD}(3)$. The SPD matrix is visualized in Figure~\ref{f:SPDdata}(a) by the corresponding ellipsoid at each pixel location. The noisy data in Figure~\ref{f:SPDdata}(b) is generated by adding Rician noise with standard deviation $0.03$, $\tilde f = A A^T$.
Here,  $A= R+\sqrt{0.03}B $ where $R^TR = f$ is the Cholesky factorization of $f$ and $B$ is a $3\times 3$ upper triangular matrix with each element being a random number from the standard normal distribution. Since $f$ is symmetric positive definite, $R$ is well defined in the Cholesky factorization. We note that adding noise in such way implies that $\tilde f$ is symmetric positive definite.

In this example, we take the target set to be the group of $3 \times 3$ symmetric positive definite matrices, $\textrm{SPD}(3)$, and solve the free space heat diffusion equation in Algorithm \ref{a:alg1}.
Then, we use the mapping 
$$
\Pi_T (A) = U \Sigma_+ V^T
$$
where $U \Sigma V^T$ is the singular value decomposition of $A$ and $\Sigma_+(i,j) = \max(\Sigma(i,j),0)$.

Figure~\ref{f:SPDres} displays the results of reconstruction with $\tau = 10^{-3}$ and $\lambda= 0.05$, $0.1$, and $0.15$, respectively. The denoised images are very close to the original data. Algorithm~\ref{a:alg1} is relatively insensitive to the value of the parameter $\lambda$. 
Figure~\ref{f:SPDres}(a) is slightly smoothed when $\lambda$ is relatively small and the regularity term dominates.
 
\begin{figure}[t]
\centering
\subfigure[Original image.]{\includegraphics[width=.45\textwidth,clip,trim= 6cm 1cm 6cm 1cm]{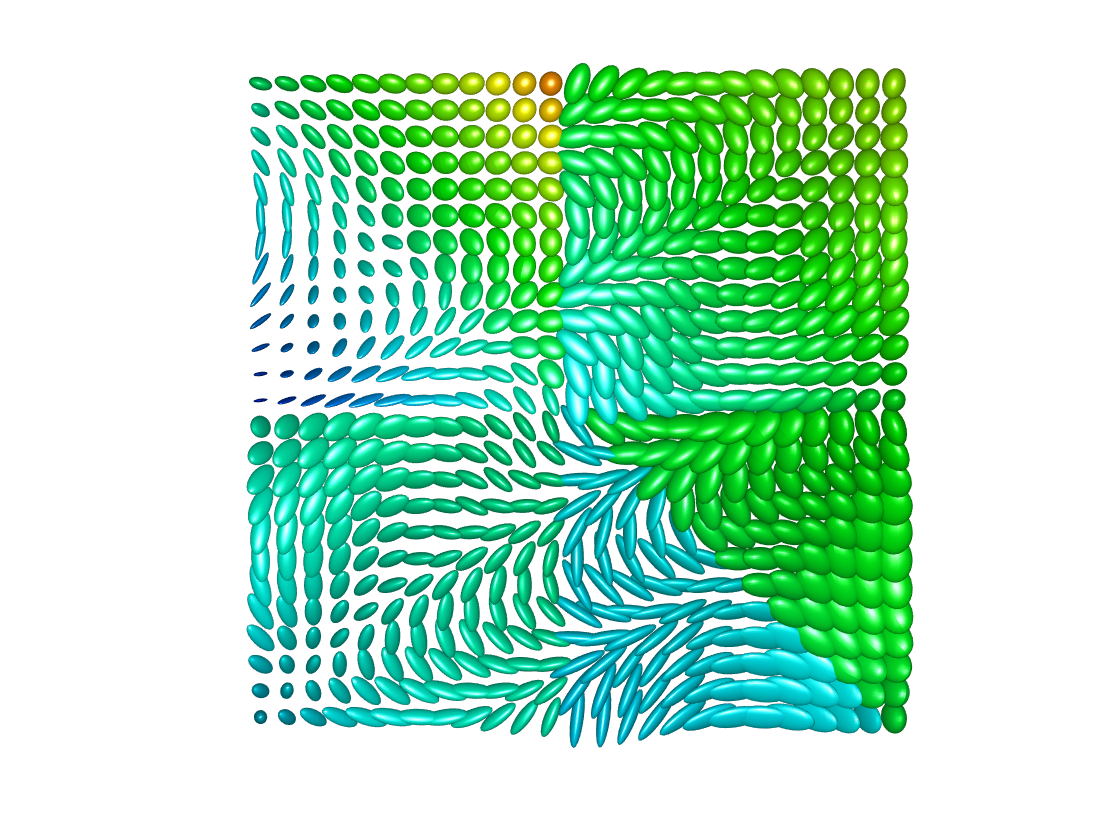}}
\subfigure[Noisy image.]{\includegraphics[width=.45\textwidth,clip,trim= 6cm 1cm 6cm 1cm]{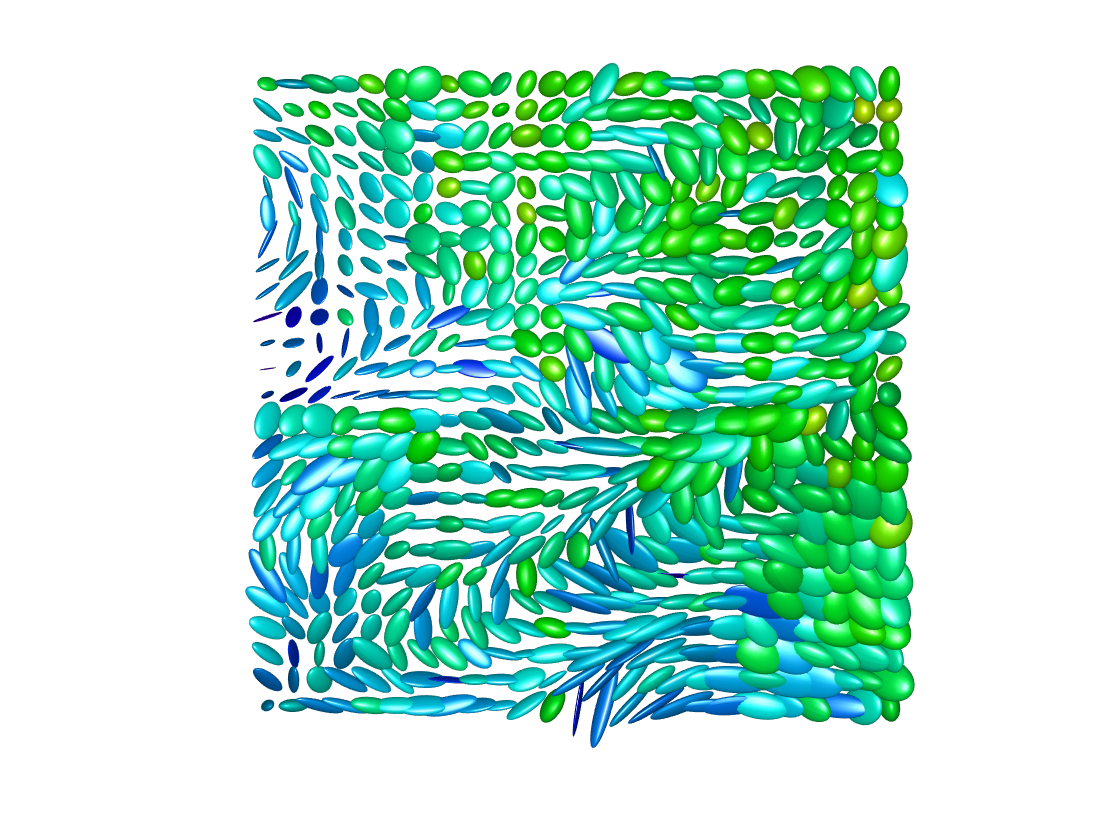}}
\caption{The original and noisy SPD-valued image. See Section~\ref{s:SPD}.}\label{f:SPDdata}
\end{figure}

\begin{figure}[t]
\centering
\subfigure[$\tau = 10^{-3}$ and $\lambda = 0.05$.]{\includegraphics[width=.3\textwidth,clip,trim= 8cm 1cm 6cm 1cm]{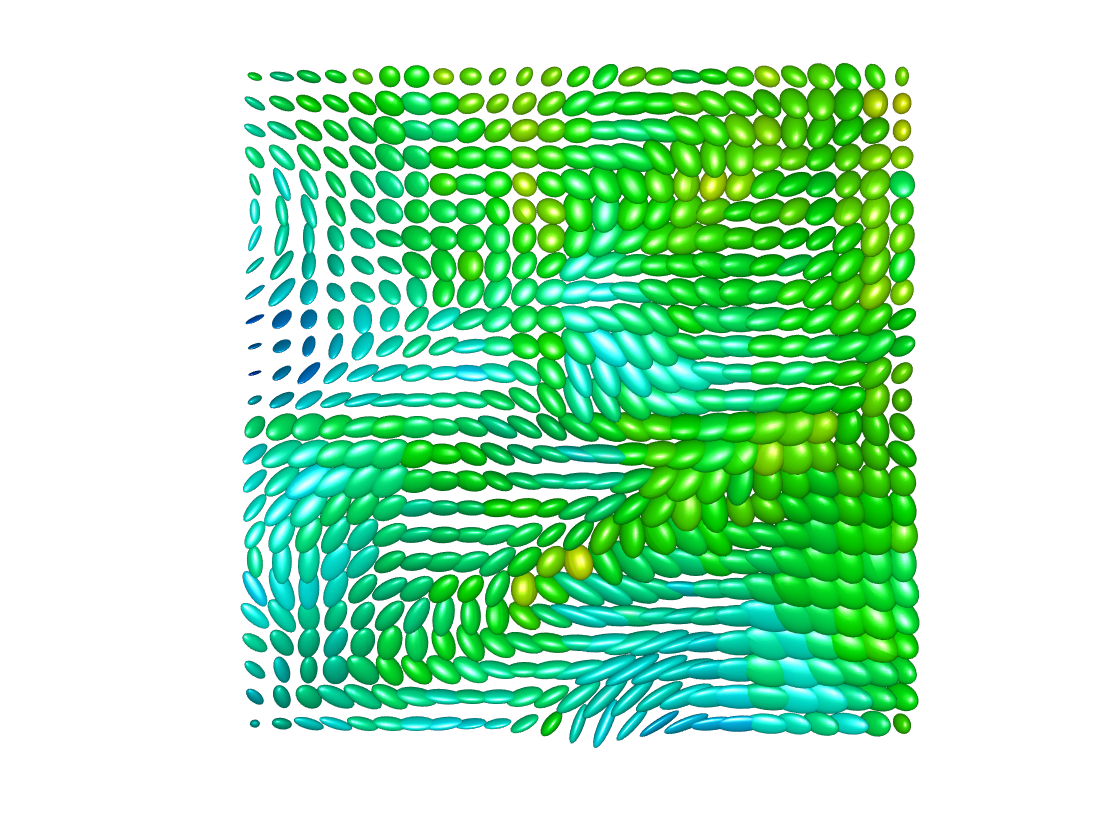}}
\subfigure[$\tau = 10^{-3}$ and $\lambda = 0.1$.]{\includegraphics[width=.3\textwidth,clip,trim= 8cm 1cm 6cm 1cm]{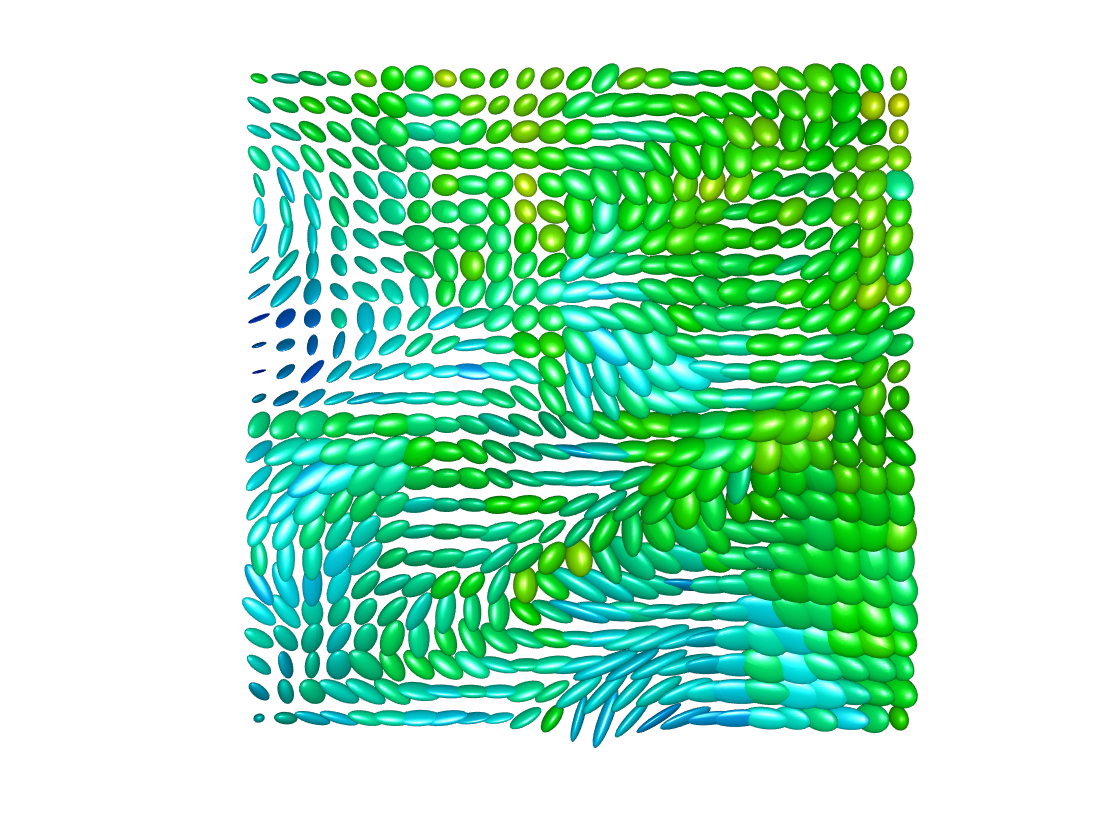}}
\subfigure[$\tau = 10^{-3}$ and $\lambda = 0.15$.]{\includegraphics[width=.3\textwidth,clip,trim= 8cm 1cm 6cm 1cm]{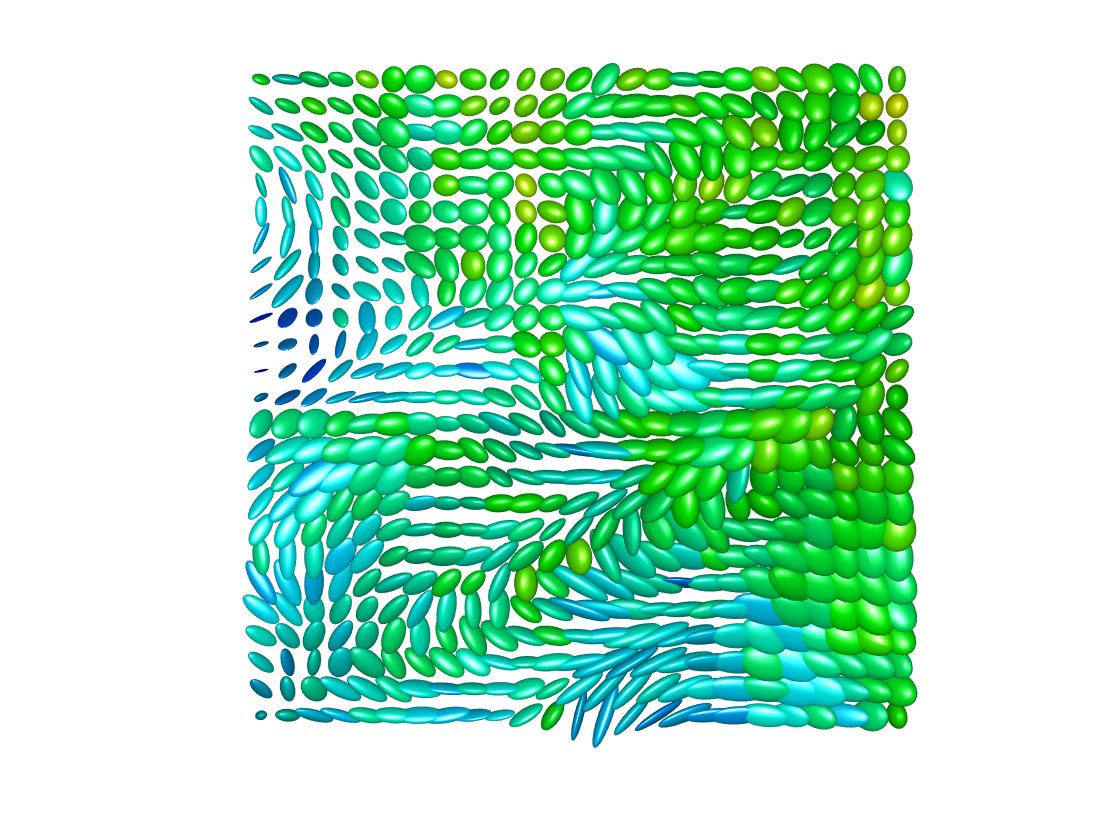}}
\caption{The result of reconstruction  in Figure~\ref{f:SPDdata}(b) with $\tau = 10^{-3}$ and $\lambda = 0.05$, $0.1$, and $0.15$. See Section~\ref{s:SPD}.}\label{f:SPDres}
\end{figure}

\subsection{Example: DT-MRI} \label{s:MRI}
Following \cite[\S5.2]{Bacak_2016}, we apply Algorithm~\ref{a:alg1} to a dataset from the Camino project\footnote{See \url{http://camino.cs.ucl.ac.uk/}.} of a diffusion tensor magnetic resonance image (DT-MRI) of the human head.  From the complete data set of $f \colon \Omega \to \{\mathcal{P}(3)\}^{112\times 112 \times 50}$ where $\Omega = [112] \times [112] \times [50]$,  we take the 28-th traversal plane, $\Omega_0 = [112] \times [112] \times \{28\} \subset \Omega$ as the dataset for reconstruction (See Figure~\ref{f:MRT_data}(a)) and zoom-in on the subset $\Omega_{1} = \{28, \ldots, 87\} \times \{24,\ldots, 73\}\times\{28\}\subset\Omega_0$ in Figure~\ref{f:MRT_data}(b).

As in the example in Section~\ref{s:SPD}, we take the target set to be the group of $3 \times 3$ symmetric positive definite matrices, $\mathrm{SPD}(3)$, and solve the free space heat diffusion equation in Algorithm \ref{a:alg1}. Figure~\ref{f:MRT_res} displays the reconstructions obtained using Algorithm~\ref{a:alg1} of the data in Figure~\ref{f:MRT_data} with $\tau = 10^{-4}$ and $\lambda = 0.1$, $0.2$, and $0.3$. The first row displays the reconstructed data in $\Omega_0$ and the second row displays the subset $\Omega_{1}$ of the corresponding reconstructed data. The results are very similar to those in \cite{Bacak_2016}. 

\begin{figure}[t]
\begin{center}
\subfigure[Slice 28 of the original Camino DT-MRI data, $\Omega_0$.]{\includegraphics[width=.5\textwidth]{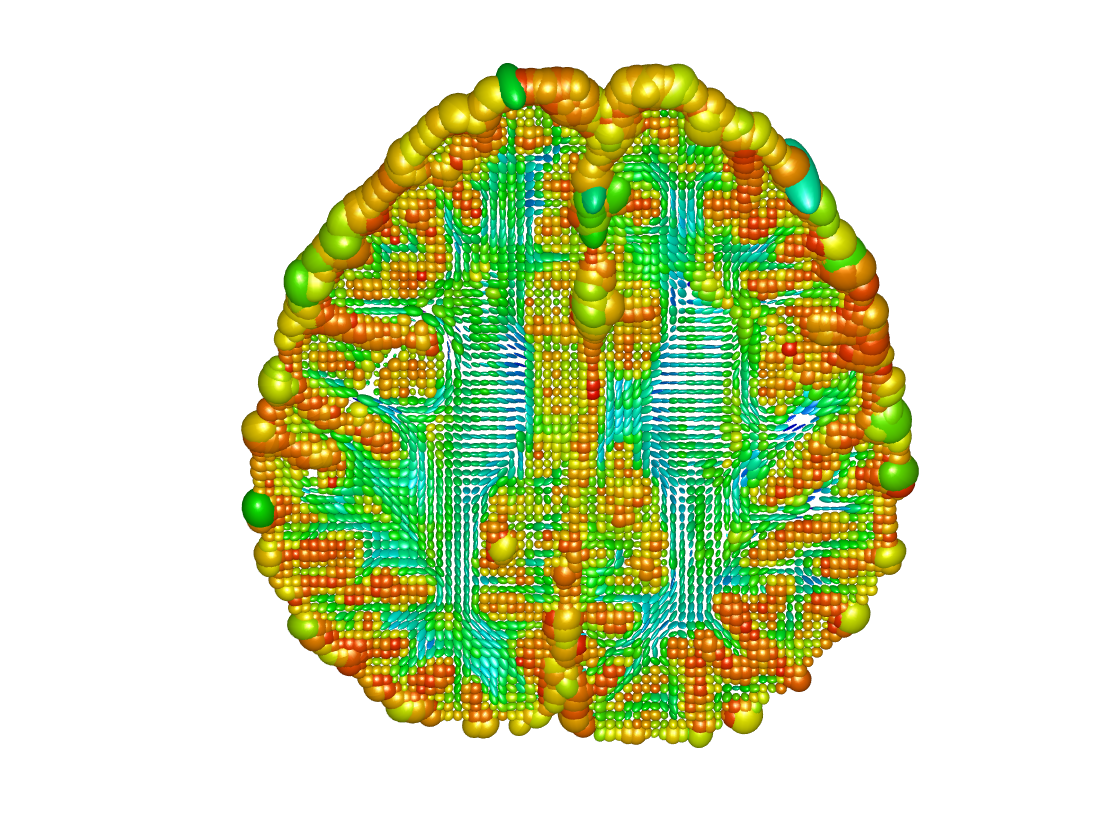}}
\subfigure[Subset $\Omega_{1} \subset \Omega_0$.]{\includegraphics[width=.4\textwidth,clip,trim= 3cm 2cm 3cm 2cm]{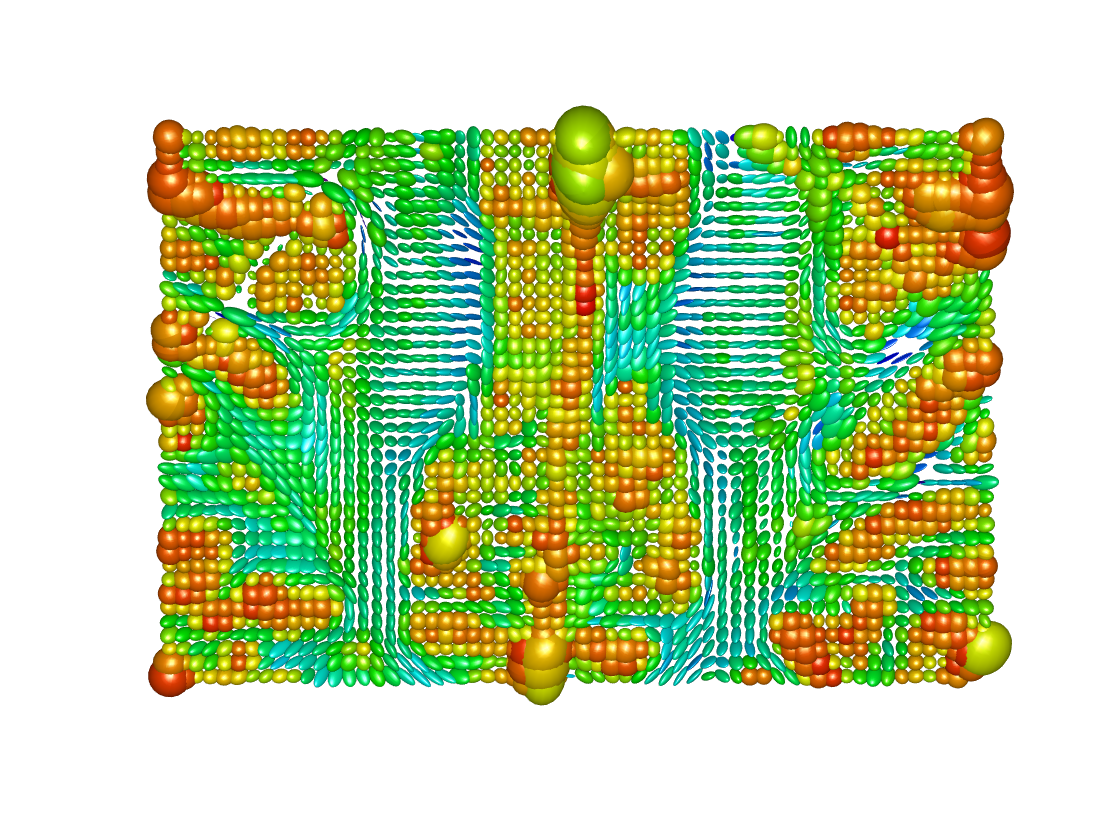}}
\caption{Slice 28 of the original Camino DT-MRI data and the subset $\Omega_{1}$.  See Section~\ref{s:MRI}.}\label{f:MRT_data}
\end{center}
\end{figure}

\begin{figure}[t]
\begin{center}
\subfigure[$\tau = 10^{-4}$ and $\lambda= 0.1$.]{\includegraphics[width=.3\textwidth,clip,trim= 7cm 4cm 7cm 4cm]{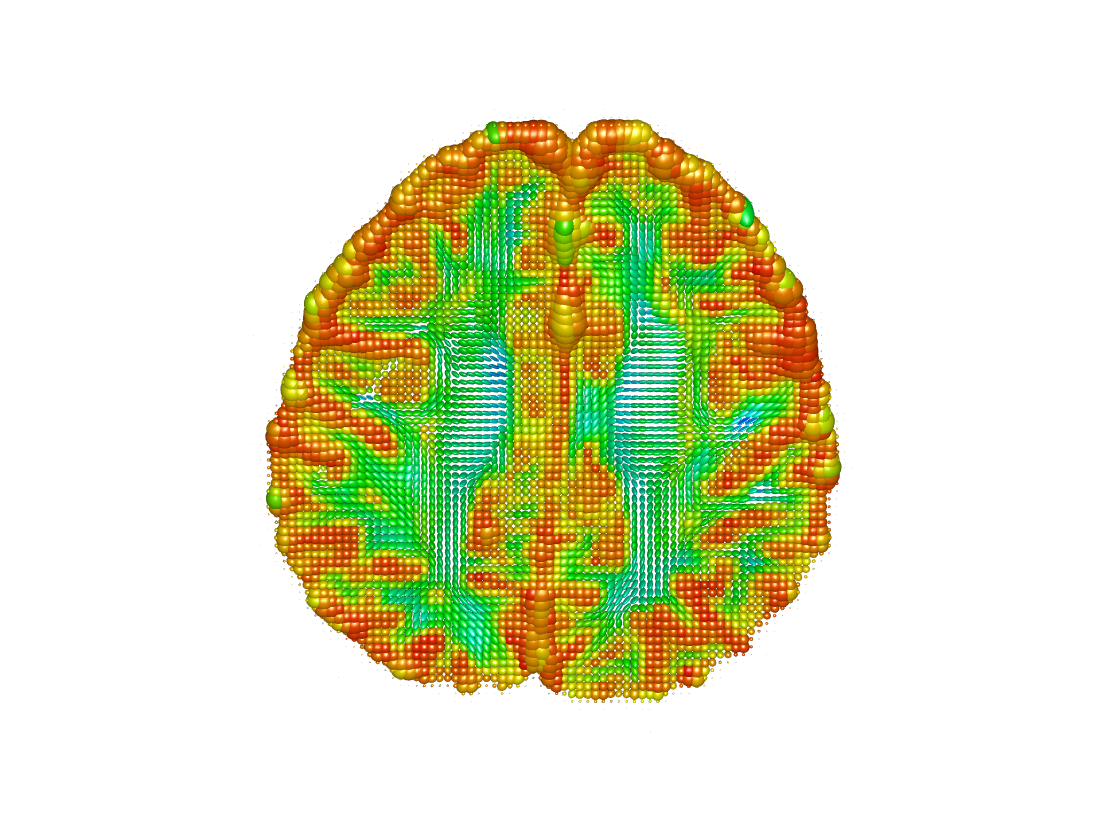}}
\subfigure[$\tau = 10^{-4}$ and $\lambda= 0.2$.]{\includegraphics[width=.3\textwidth,clip,trim= 7cm 4cm 7cm 4cm]{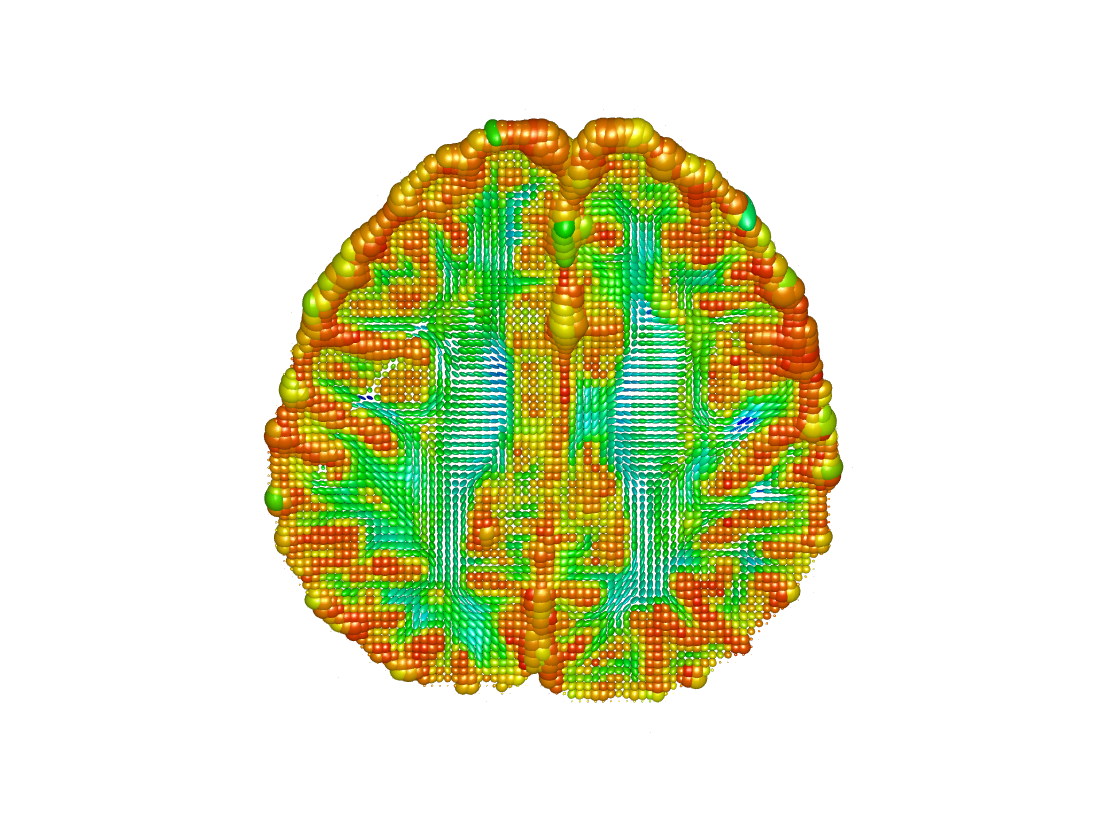}} 
\subfigure[$\tau = 10^{-4}$ and $\lambda= 0.3$.]{\includegraphics[width=.3\textwidth,clip,trim= 7cm 4cm 7cm 4cm]{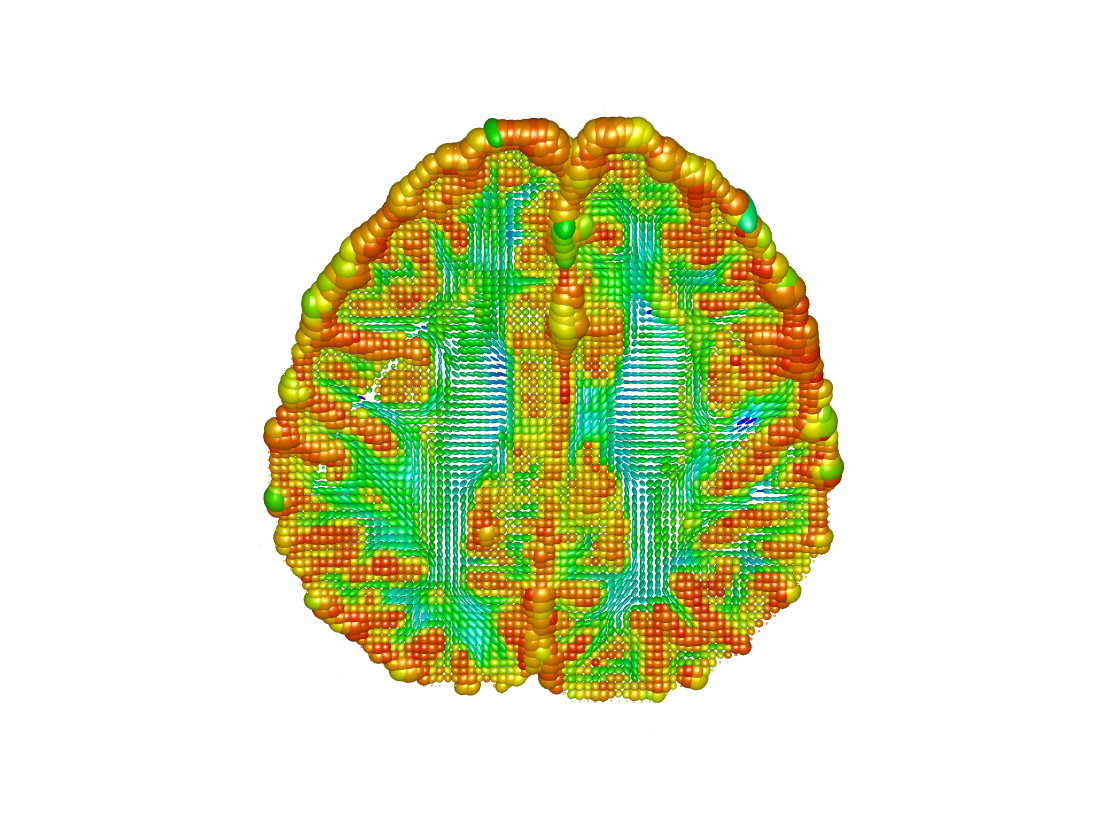}}\\
\subfigure[$\Omega_{1} \subset \Omega_0$ in (a).]{\includegraphics[width=.3\textwidth,clip,trim= 3cm 2cm 3cm 2cm]{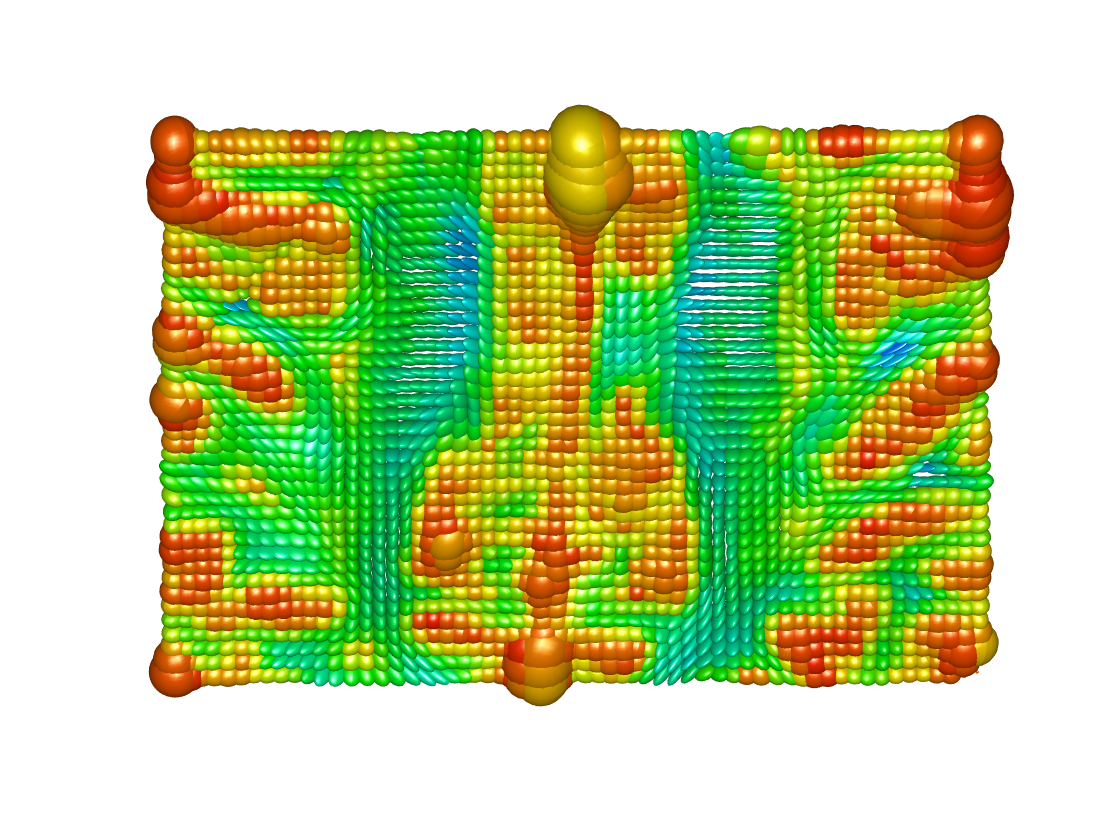}}
\subfigure[$\Omega_{1} \subset \Omega_0$ in (b).]{\includegraphics[width=.3\textwidth,clip,trim= 3cm 2cm 3cm 2cm]{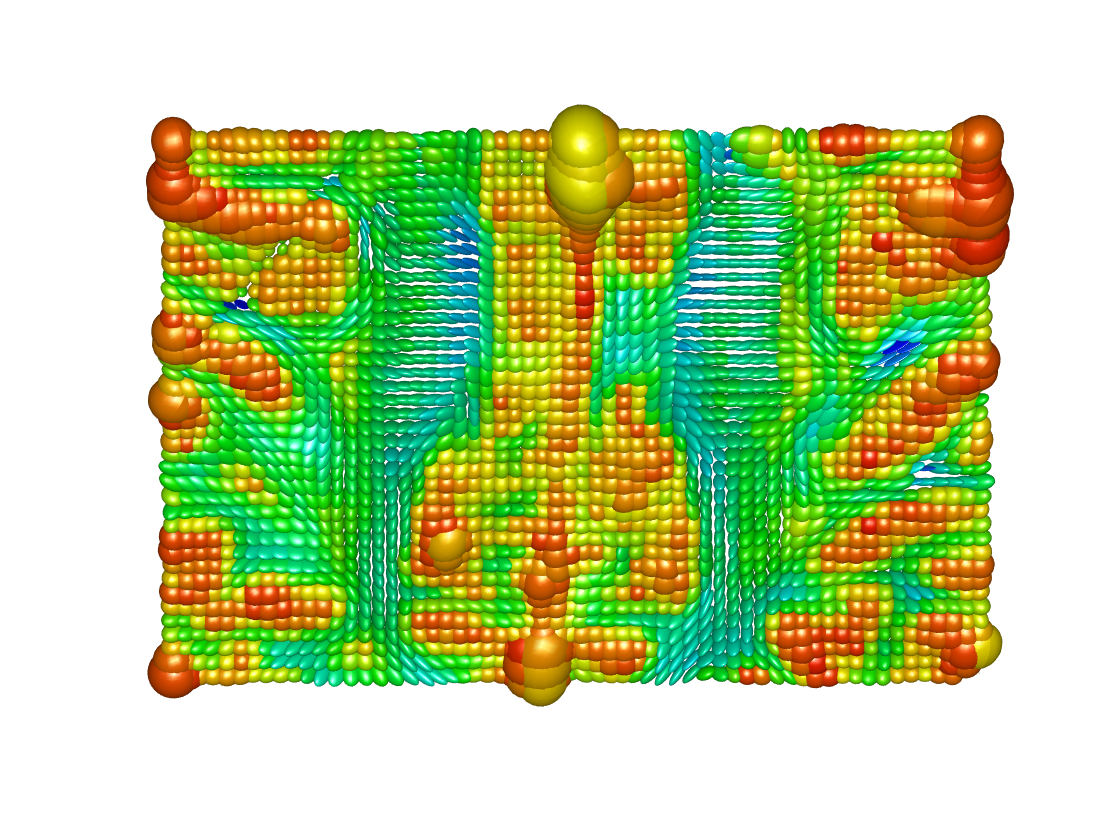}}
\subfigure[$\Omega_{1} \subset \Omega_0$ in (c).]{\includegraphics[width=.3\textwidth,clip,trim= 3cm 2cm 3cm 2cm]{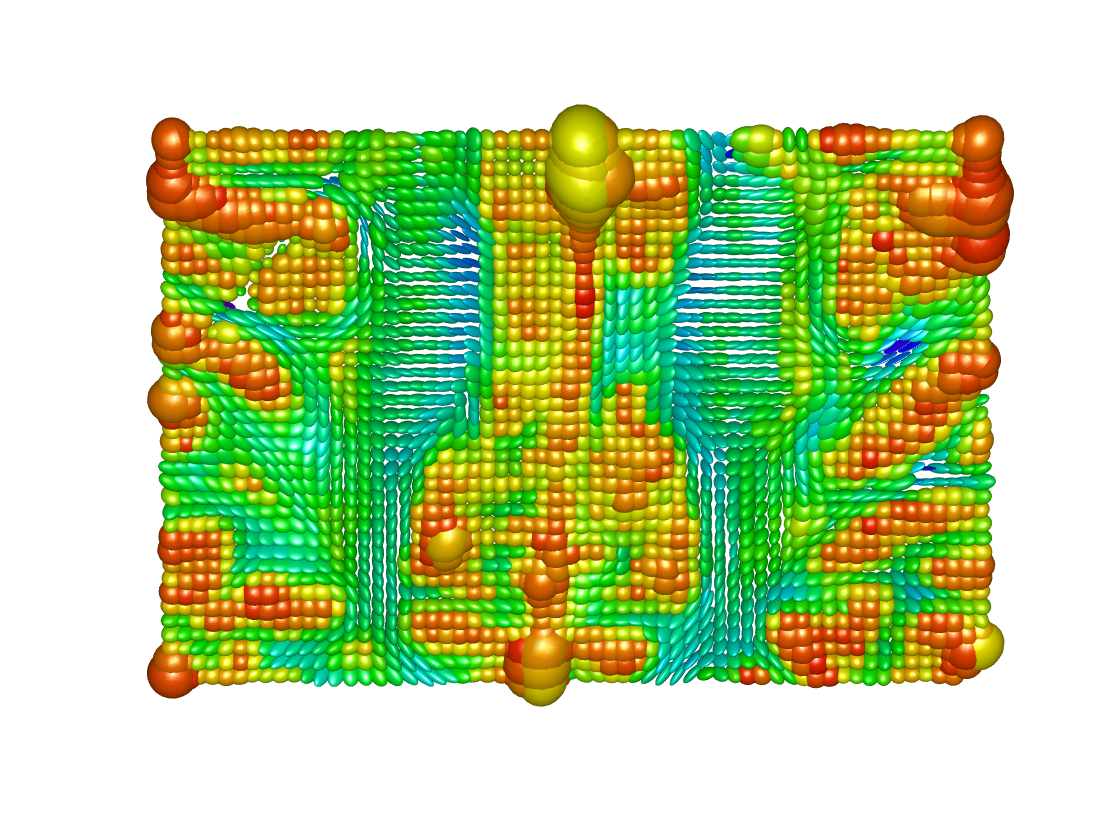}}
\caption{ Results of the reconstruction in Figure~\ref{f:MRT_data}(a) with $\tau = 10^{-4}$ and $\lambda = 0.1$, $0.2$, and $0.3$. See Section~\ref{s:MRI}.}\label{f:MRT_res}
\end{center}
\end{figure}

\subsection{Example: an $\mathbb{R}\mathbb{P}^1$-valued image}\label{s:rpk}
In this section, we denoise a real projective line ($\mathbb{R}\mathbb{P}^1$)-valued image. It is not clear that real projective spaces, $\mathbb{R}\mathbb{P}^{n}$, for $n>1$ can be cast within the framework of the proposed methods. However, due to the topological equivalence of $\mathbb{RP}^1$ with the circle, $\mathbb S^1$, we can study real projection line-valued images as follows.

We recall that $\mathbb{R}\mathbb{P}^{n-1}$ can be viewed as identifying antipodal points of the unit $n$-sphere, $\mathbb S^{n-1} \subset \mathbb R^k$. 
Thus, identifying $\mathbb R^2$ with $\mathbb C$, we can uniquely represent each element $\{x,-x\}\in \mathbb{RP}^1$, by its squared value, $x^2 \in \mathbb S^1 \subset \mathbb C$. We then denoise this representation map using Algorithm \ref{a:alg1} and the mapping $\Pi_T = \frac{x}{|x|}$. Finally, by taking the two square roots of the denoised representation map, we obtain a denoised $\mathbb{RP}^1$-valued image.   A similar strategy is employed in \cite{Viertel2017} for cross-valued fields. 

We define a synthetic $\mathbb{RP}^1$-valued image as follows. 
First define $z(x,y) = \left(\Re(\sqrt{i x - y}),\Im(\sqrt{i x - y})\right)$ where $\Re(\cdot)$ and $\Im(\cdot)$ denote the real and imaginary parts.
Now we define the line field, 
\[f(x,y) = \left \{  \frac{ [z(x,y)]_+}{|z(x,y)}, \ \frac{ [z(x,y)]_-}{|z(x,y)} \right\} \in \mathbb{RP}^1, 
\qquad \qquad \textrm{where} \ x,y \in [-1,1]. \]
Here, $[z]_+ = (|z_1|, z_2) \in S^1 \cap \{z\colon z_1>0\}$ and $[z]_+ = - [z]_+$. 
We discretize the parameter space $(x,y)$ by $20\times 20$ grid points to give a discretized line field, $\{f_{i,j}\}_{i,j=1}^{20} \subset \mathbb{RP}^1$, 
which is displayed in  Figure~\ref{f:rp2data}(a). 
We then add Gaussian noise with standard deviation $0.3$ to $(x_i,y_i)$ point- and component-wise to have $(\tilde{x}_i, \tilde{y}_i)$. Noisy data is then generated by $\tilde f_{i,j} = f(\tilde x_i,\tilde y_j)$ and plotted in Figure~\ref{f:rp2data}(b). 

Proceeding as described above, we view the noisy line field as taking values in $\mathbb C$ and point-wise square the values to obtain the representation field, $\tilde f^2 \subset \mathbb S^1$. We apply Algorithm~\ref{a:alg1} to  $\tilde f^2$  with target set, $T= \mathbb{S}^1$. We then take the point-wise square root of the resulting field to obtain the denoised line field. The results are displayed in Figure~\ref{f:rp2res} with $\tau = 10^{-2}$ and $\lambda = 0.05$, $0.1$, and $0.15$. We observe that the reconstructed images are very close to the original data and the index $+1/2$ singularity is well-preserved. 

\begin{figure}[t]
\centering
\subfigure[Original image.]{\includegraphics[width=.35\textwidth,clip,trim= 6cm 2cm 6cm 2cm]{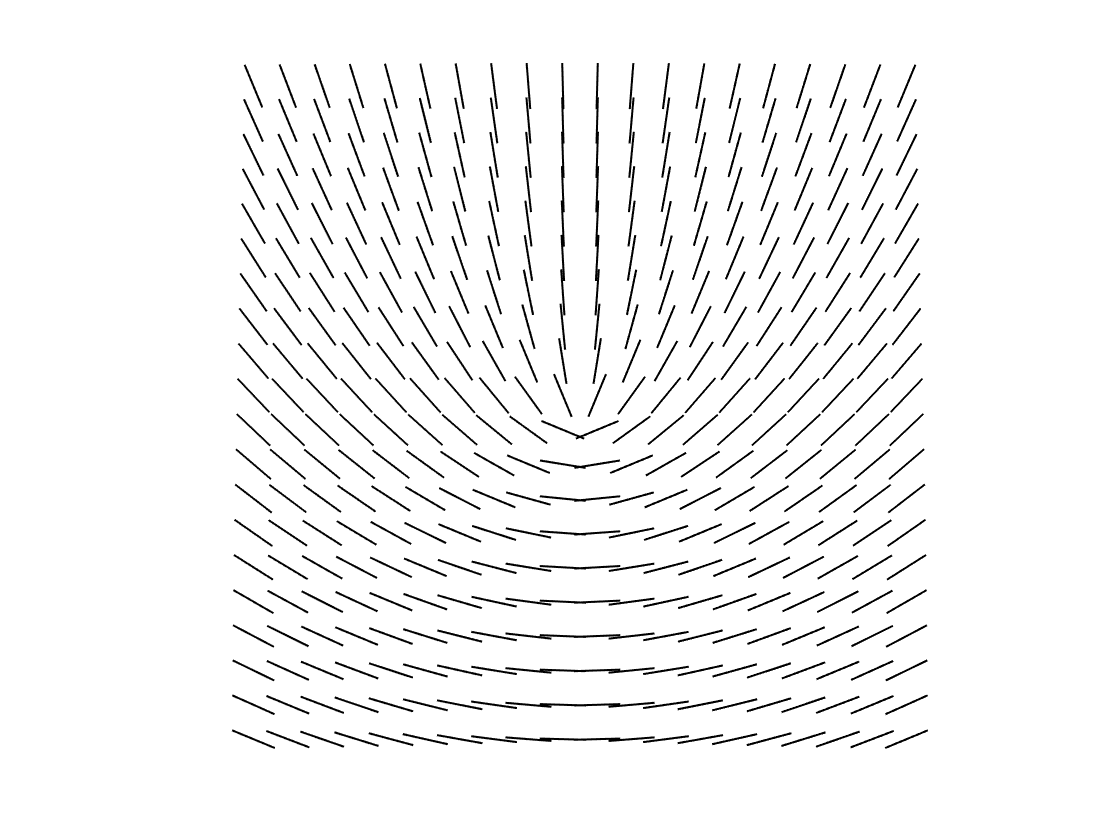}}
\subfigure[Noisy image.]{\includegraphics[width=.35\textwidth,clip,trim= 6cm 2cm 6cm 2cm]{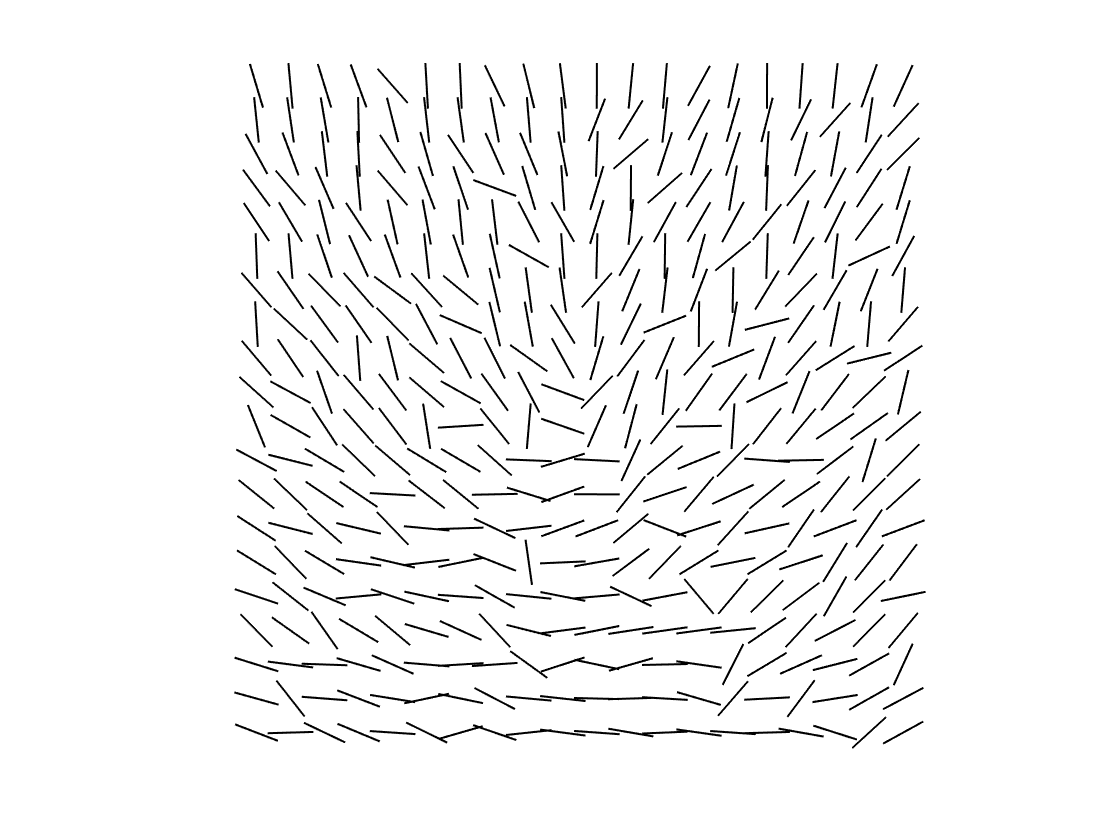}}
\caption{The original and noisy synthetic line-field image.  See Section~\ref{s:rpk}.}\label{f:rp2data}
\end{figure}

\begin{figure}[t]
\centering
\subfigure[$\tau=10^{-2}$ and $\lambda = 0.05$.]{\includegraphics[width=.3\textwidth,clip,trim= 6cm 3cm 6cm 2cm]{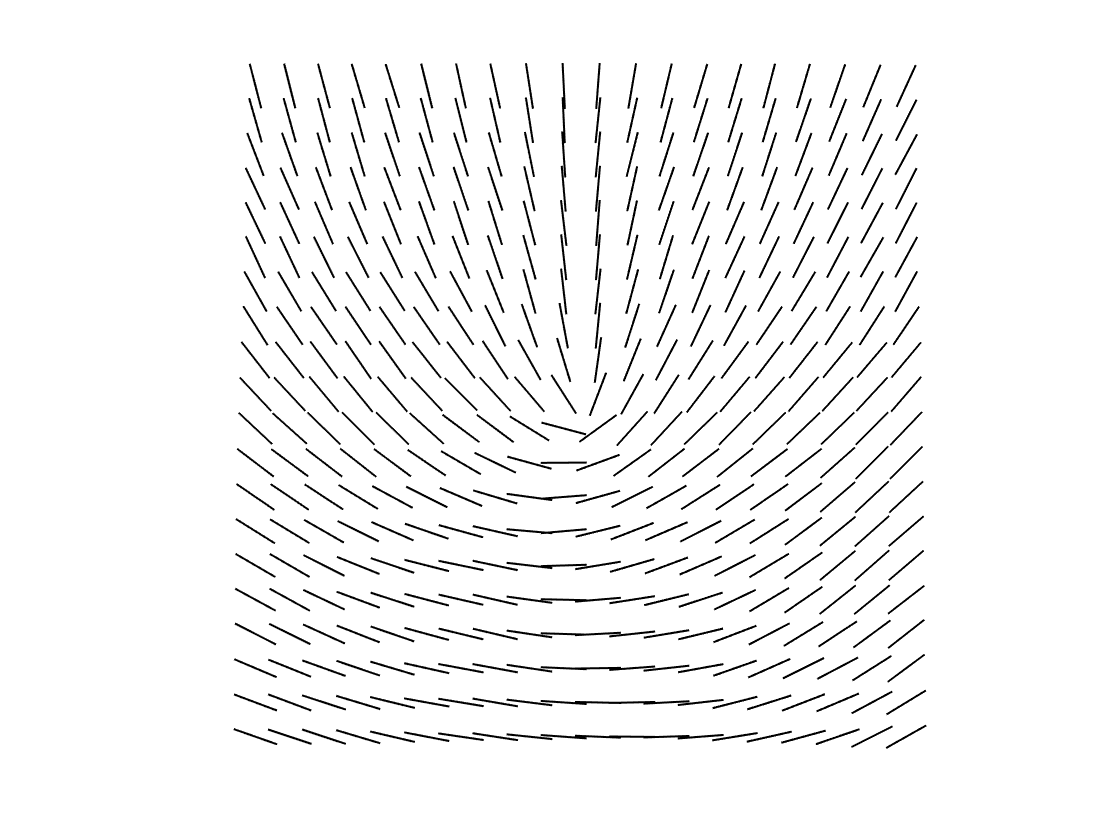}}
\subfigure[$\tau = 10^{-2}$ and $\lambda = 0.1$.]{\includegraphics[width=.3\textwidth,clip,trim= 6cm 3cm 6cm 2cm]{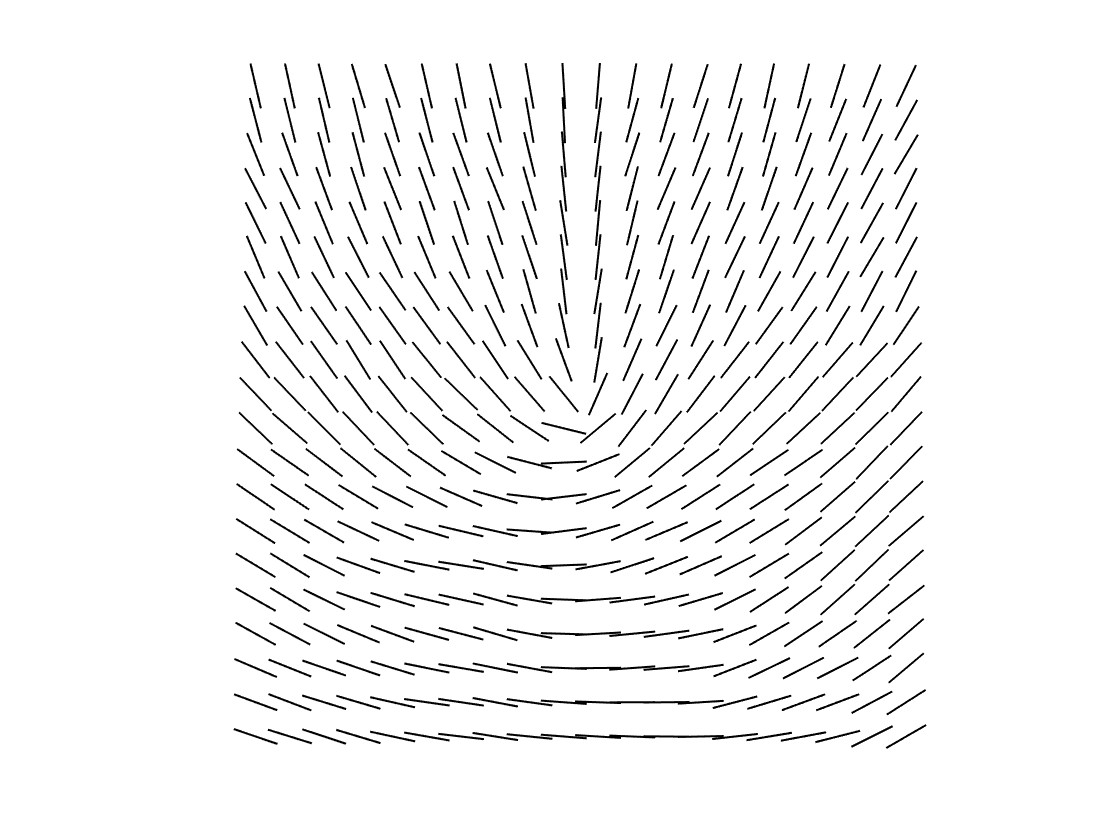}}
\subfigure[$\tau = 10^{-2}$ and $\lambda = 0.15$.]{\includegraphics[width=.3\textwidth,clip,trim= 6cm 3cm 6cm 2cm]{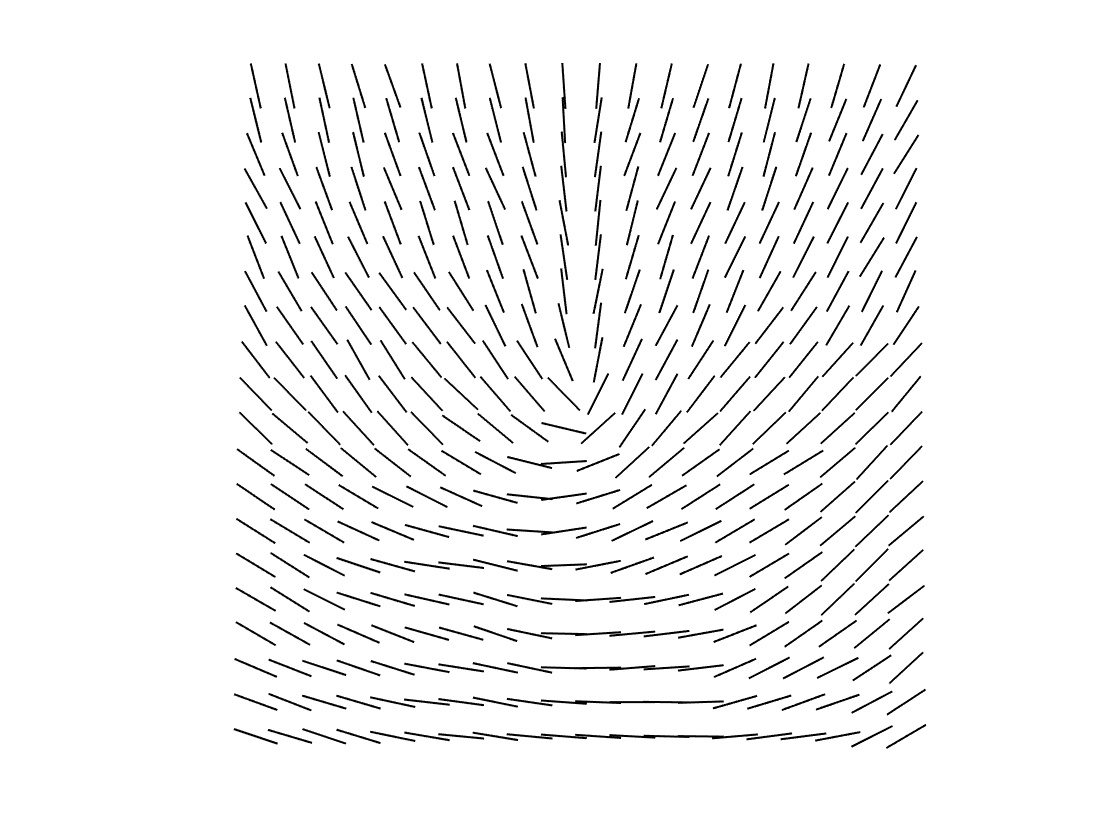}}
\caption{ Denoising results for the line-field in Figure~\ref{f:rp2data}(b) with $\tau = 10^{-2}$ and $\lambda = 0.05$, $0.1$, and $0.15$. See Section~\ref{s:rpk}.}\label{f:rp2res}
\end{figure}

\subsection{Example: a fingerprint image}\label{s:finger}
In this section, using the methods described in Section \ref{s:rpk} for $\mathbb{RP}^1$,  we analyze an image of a fingerprint;  see Figures~\ref{f:finger}(a) and (d). 
From the fingerprint, we extract a very rough line field, $\{f_{ij}\} \subset \mathbb{RP}^1$; see Figures~\ref{f:finger}(b) and (e). 
We apply Algorithm \ref{a:alg1} on the `squared field' $\{f_{ij}^2\}$ and display the denoised results in Figures~\ref{f:finger}(c) and \ref{f:finger}(f). 
We observe that the denoised line field is a good model of the original fingerprint. 
In Figure~\ref{f:finger}(c), the blue dot indicates a singularity with index $1/2$. 
In Figure~\ref{f:finger}(f), there are two singularities: the one indicated by the blue dot has index $1$ and the one indicated by the green dot has index $-1/2$. 

In both experiments, we choose $\tau = 10^{-2}$ and $\lambda = 0.15$.  We extracted the line field at $38 \times 28$ points for \ref{f:finger}(a) and $35 \times 25$ points for \ref{f:finger}(d).  Both simulations were done in $5\times 10^{-3}$ seconds, which demonstrates the efficiency of the proposed algorithm. 

\begin{figure}[t]
\centering
\subfigure[Original fingerprint.]{\includegraphics[width=.3\textwidth,clip,trim= 0cm 0cm 0cm 0cm]{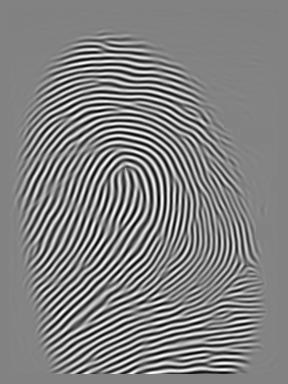}}
\subfigure[Noisy orientation field on the fingerprint (a).]{\includegraphics[width=.3\textwidth,clip,trim= 0cm 0cm 0cm 0cm]{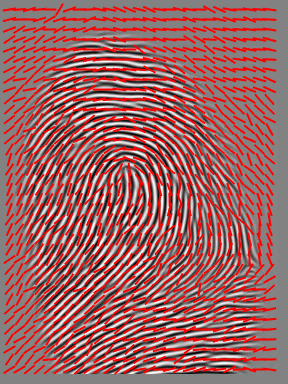}}
\subfigure[$\tau = 10^{-2}$ and $\lambda = 0.15$.]{\includegraphics[width=.3\textwidth,clip,trim= 0cm 0cm 0cm 0cm]{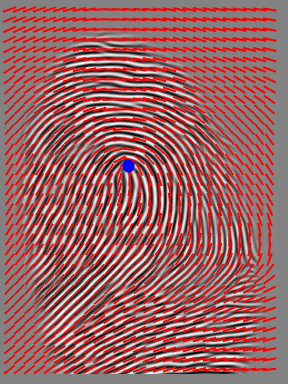}}
\subfigure[Original fingerprint.]{\includegraphics[width=.3\textwidth,clip,trim= 0cm 0cm 0cm 0cm]{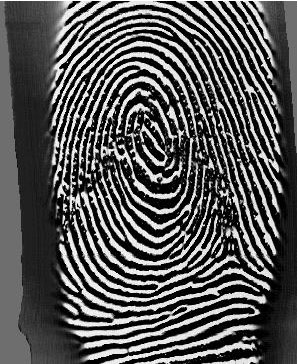}}
\subfigure[Noisy orientation field on the fingerprint (d).]{\includegraphics[width=.3\textwidth,clip,trim= 0cm 0cm 0cm 0cm]{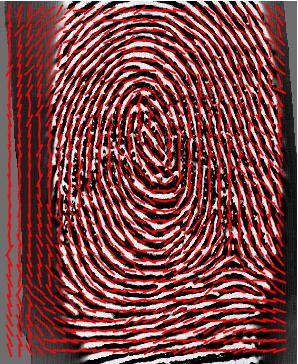}}
\subfigure[$\tau = 10^{-2}$ and $\lambda = 0.15$.]{\includegraphics[width=.3\textwidth,clip,trim= 0cm 0cm 0cm 0cm]{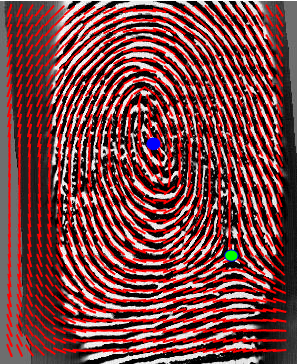}}
\caption{ Denoising results for two fingerprint images with $\tau = 10^{-2}$ and $\lambda = 0.15$. See Section~\ref{s:finger}.}\label{f:finger}
\end{figure}

\section{Discussion} \label{s:Disc}
In this paper, we introduced and analyzed a nonlocal energy for denoising target-valued images. We derived a diffusion generated method to minimize the energy and performed a variety of numerical experiments to show that the method is efficient, stable, and applicable to a wide variety of target-sets. There are a variety of interesting future directions for this work. 

The closest comparison for our numerical results can be found in \cite{Weinmann_2014,Bacak_2016}. The models developed in these papers are nonsmooth variational models which include total variation or second order differences in the regularization term. While we expect that these methods preserve edges better than the proposed method, the results in the numerical experiments are visually very similar. However, due to the simplicity of viewing the target set in an ambient Euclidean space, our methods should be faster. As with any inverse problem, the `best' method depends on the structure of the image and the noise as well as the size of the data. More work should be done to understand the statistical framework for which these methods are consistent and robust estimators. 

In this method, we have taken the domain, $\Omega$, to be a Euclidean set. It would be very interesting to consider the case when $\Omega$ is a graph and the energy \eqref{e:relax2} is formulated using the analogous graph operators \cite{Gennip2013,Bergmann_2018}.

In this work, we have only looked at the denoising problem for target-valued images. Other image analysis tasks for target-valued images, including inpainting, segmentation, and registration, could be handled using similar techniques.

\printbibliography

\end{document}